\documentclass[10pt]{article}
\usepackage{fullpage}
\usepackage{amsmath,amssymb,amsthm,mathrsfs,float,stmaryrd}
\usepackage[PostScript=dvips]{diagrams}
\usepackage[pdftex]{hyperref}

\newcommand{\N}{\mathbb{N}}

\newcommand{\Z}{\mathbb{Z}}
\newcommand{\Q}{\mathbb{Q}}
\newcommand{\C}{\mathbb{C}}
\newcommand{\G}{\mathbb{G}}
\newcommand{\CP}{\C P}
\newcommand{\bundleL}{\mathbb{L}}
\newcommand{\torsorA}{\mathscr{A}}
\newcommand{\torsorB}{\mathscr{B}}
\newcommand{\torsorC}{\mathscr{C}}
\newcommand{\into}{\hookrightarrow}

\newcommand{\onto}{\twoheadrightarrow}

\DeclareMathOperator{\mspan}{span}
\DeclareMathOperator{\ann}{Ann}

\DeclareMathOperator{\spec}{spec}

\DeclareMathOperator*{\Ext}{Ext}
\DeclareMathOperator{\Hom}{Hom}
\newcommand{\coboundarymap}{\delta}

\newcommand{\weight}[1]{|#1|}
\newcommand{\length}[1]{\ell(#1)}
\newcommand{\drawbox}[1]{\begin{array}{|l|}\hline #1 \\ \hline\end{array}}
\newarrow{Gather}   o--->
\newarrow{Split}    {triangle}--->
\newarrow{AllDash}  {}{dash}{}{dash}{}

\makeatletter
\def\@seccntformat#1{\csname the#1\endcsname.\quad}
\makeatother
\newtheoremstyle{thmstyle}{\topsep}{\topsep}{}{}{\bf}{:}{.5em}{}
\theoremstyle{thmstyle}
\newtheorem{thm}{Theorem}[subsection]
\newtheorem{defn}[thm]{Definition}
\newtheorem{cor}[thm]{Corollary}
\newtheorem{lem}[thm]{Lemma}
\newtheorem{rem}[thm]{Remark}
\newtheorem{ithm}{Theorem}[section]

\title{An Elementary Classification of Symmetric $2$-Cocycles \footnote{The authors were supported by NSF grant DMS-0705233.}}
\author{Adam Hughes, JohnMark Lau, Eric Peterson
}
\hypersetup{
        pdfauthor={Adam Hughes, JohnMark Lau, Eric Peterson
        },
        pdftitle={An Elementary Classification of Symmetric 2-Cocycles},
        bookmarks}

\begin{document}


\maketitle

\begin{abstract}
We present a classification of the so-called ``additive symmetric 2-cocycles'' of arbitrary degree and dimension over $\Z_p$, along with a partial result and some conjectures for $m$-cocycles over $\Z_p$, $m > 2$.  This expands greatly on a result originally due to Lazard and more recently investigated by Ando, Hopkins, and Strickland, which together with their work culminates in a complete classification of $2$-cocycles over an arbitrary commutative ring.  The ring classifying these polynomials finds application in algebraic topology, including generalizations of formal group laws and of cubical structures.
\end{abstract}
\tableofcontents
\newpage

\section{Overview}\label{Overview}

An additive symmetric $2$-cocycle in $k$-variables over a commutative ring $A$ (or simply ``a cocycle'') is a symmetric polynomial $f \in A[x_1, \ldots, x_k]$ that satisfies the following equation: \[f(x_1, x_2, x_3, \ldots, x_k) - f(x_0 + x_1, x_2, x_3, \ldots, x_k) + f(x_0, x_1 + x_2, x_3, \ldots, x_k) - f(x_0, x_1, x_3, \ldots, x_k) = 0.\]  When $k = 2$, these polynomials were classified by Lazard in \cite{Laz55} in the context of formal group laws, where he exhibited a countable basis for the space of cocycles of the form \[f_n(x, y) = \left(\gcd_{0 < i < n}{n \choose i}\right)^{-1} \left((x + y)^n - x^n - y^n\right),\] one for each $n \in \N$.

An extension of Lazard's work was considered by Ando, Hopkins, and Strickland in \cite{AHS01} to explore $BU \langle 2k \rangle$ (see \textsection\ref{TheFunctorspecHBU2k}).  They accomplished a complete classification of the $k$-variable rational cocycles for all $k$, where they found that they were generated by a unique polynomial in each homogenous degree given by \[\zeta_k^n = d^{-1} \sum_{\substack{I \subseteq \{1, \ldots, k\} \\ I \ne \emptyset}} (-1)^{|I|} \left( \sum_{i \in I} x_i \right)^n,\] where $d$ is the gcd of the coefficients of the right-hand sum.  The form of these cocycles is a relatively straightforward generalization of Lazard's cocycles; note that \[\zeta_2^n = d^{-1} \sum_{\substack{I \subseteq \{1, 2\} \\ I \ne \emptyset}} (-1)^{|I|} \left( \sum_{i \in I} x_I \right)^n = d^{-1} \left( (x_1 + x_2)^n - x_1^n - x_2^n \right) = f_n(x_1, x_2).\]  The authors also found a classification for $A$-cocycles in three-variable case for any commutative ring $A$.  In the particular case when $A = \Z_p$, they found generators in each homogenous degree given by one or both of $\zeta_3^n$ and $(\zeta_3^{n / p})^p$ under the projection $\pi_p: \Z \onto \Z_p$, the second considered precisely when $p \mid n$.

This modular classification is what we complete for higher $k$.  What separates our approach from past classifications is that we construct the classification for all $k$ in concert; indeed, the classification for $k$ variable often depends upon the classification for $r$ variables, $r > k$.

First, to each integer partition $\lambda$ of $n$, we associate a symmetric polynomial $\tau \lambda$ given by \[\tau \lambda = d^{-1} \sum_{\sigma \in S_k} x_1^{\lambda_{\sigma 1}} x_2^{\lambda_{\sigma 2}} \cdots x_k^{\lambda_{\sigma k}} \in \Z[\mathbf{x}],\] \[d = |\{\sigma \in S_k \mid \sigma \lambda = \lambda\}|,\] where $S_k$ acts on an ordered partition $\lambda$ by permuting its elements.  For instance, we have
\begin{align*}
\tau(2, 1, 1) & = x_1^2 x_2 x_3 + x_1 x_2^2 x_3 + x_1 x_2 x_3^2, \\
\tau(2, 2, 2) & = x_1^2 x_2^2 x_3^2, \\
\tau(1, 2, 3) & = x_1 x_2^2 x_3^3 + x_1 x_2^3 x_3^2 + x_1^2 x_2 x_3^3 + x_1^3 x_2 x_3^2 + x_1^2 x_2^3 x_3 + x_1^3 x_2^2 x_3.
\end{align*}
We say such a partition is power-of-$p$ when all its entries are integer powers of $p$.  It's not difficult to show that $\tau \lambda$ is a cocycle over $\Z_p$ when $\lambda$ is power-of-$p$.  The first step in our classification is then
\begin{ithm}\label{IntroPOPCocycles}
Let $n, k$ be such that a power-of-$p$ partition of $n$ of length $k$ exists.  Then the symmetrized monomials corresponding to power-of-$p$ partitions of $n$ of length $k$ are the \emph{only} $2$-cocycles of that homogenous degree, number of variables, and characteristic. 
\end{ithm}

This alone gives the vast majority of the classification in $\Z_2$; if for a power-of-$2$ partition $\lambda$ we can pick $x \in \lambda$ not equal to one, then we can construct the partition $\lambda' = (\lambda \setminus (x)) \cup (2^{-1} x, 2^{-1} x)$, where $\cup$ denotes partition concatenation and $\setminus$ denotes deletion.  $\lambda'$ is a power-of-$2$ partition of length one greater than $\lambda$, and we can then simply apply \ref{IntroPOPCocycles} again.  Of course, when $n = 13$, our smallest power-of-$2$ partition is given by $(8, 4, 1)$, and so \ref{IntroPOPCocycles} tells us nothing about the 2-cocycles in two variables here.  This problem becomes even more exaggerated in odd prime characteristics; the partition $(9, 3)$ gives rise to the power-of-$3$ partitions $(9, 1, 1, 1)$ and $(3, 3, 3, 3)$ by a similar splitting procedure, and now we find that we've skipped over the cocycles in three variables.  To highlight the non-power-of-$p$ cases we've left undescribed, we provide the following excerpt from the table of cocycle bases over $\Z_3$ contained in \ref{Char3Table}, as obtained by raw computation:

\begin{equation*} 
\begin{array}{r|l|l|l|l|l|}
             & \hbox{dim 2} &            3 &              4 &                5 &            6 \\
\hline
\vdots       & \vdots       & \vdots       & \vdots         & \vdots           & \vdots \\
\hbox{deg 8} & \tau(6,2)+   & \tau(6,1,1)- & \tau(3,3,1,1)  & \tau(3,2,1,1,1)- & \tau(3,1,1,1,1,1) \\
             & \tau(4,4)-   & \tau(4,3,1)+ &                & \tau(4,1,1,1,1)  &             \\
             & \tau(7,1)-   & \tau(3,3,2)  &                &                  &             \\
             & \tau(5,3)    &              &                &                  &             \\
\vdots       & \vdots       & \vdots       & \vdots         & \vdots           & \vdots \\
          12 &\tau(9,3)     & \tau(6,3,3), & \tau(3,3,3,3), & \tau(6,3,1,1,1)- & \tau(3,3,3,1,1,1) \\
             &              & \tau(9,2,1)- & \tau(9,1,1,1)  & \tau(4,3,3,1,1)+ &             \\
             &              & \tau(10,1,1) &                & \tau(3,3,3,2,1)  &             \\
\vdots       & \vdots       & \vdots       & \vdots         & \vdots           & \vdots \\
\end{array}
\end{equation*} 

To explain these other entries, we define $G_{i, j}$, called a gathering operator, to act on partitions by \[G_{i, j} : \lambda \mapsto (\lambda_i + \lambda_j) \cup (\lambda \setminus (\lambda_i, \lambda_j)).\]  Following the above example, we compute
\begin{align*}
G_{1, 2}(9, 1, 1, 1) & = (10, 1, 1), \\
G_{2, 3}(9, 1, 1, 1) & = (9, 2, 1), \\
G_{1, 2}(3, 3, 3, 3) & = (6, 3, 3).
\end{align*}

Our main result is that in all degree, dimension, characteristic triples not covered by \ref{IntroPOPCocycles}, the following theorem completes the classification:

\begin{ithm}\label{IntroCarryMinAreCocycles}
Select a power-of-$p$ partition $\lambda$ of $n$ with length $k$.  Let $T^m \lambda$ denote the set of all possible partitions of the form $G_{i_1, j_i} \cdots G_{i_m, j_m} \lambda.$  Then, if $m \le p - 2$ or if $\lambda$ is the shortest power-of-$p$ partition of $n$, the polynomial \[\sum_{\mu \in T^m \lambda} c_\mu \cdot (\tau \mu)\] will be a cocycle, where $c_\mu$ is the coefficient of $\tau \mu$ in $\pi_p \zeta_{k-m}^n$.  In addition, cocycles formed in this manner give a basis for the space of modular cocycles.
\end{ithm}

First note that by setting $m = 0$, this subsumes theorem \ref{IntroPOPCocycles}.  The theorem then applies in two cases, one corresponding to a limit on the number of gathering operators we apply and another to having picked a very particular $\lambda$.  To illustrate the first case, we continue our example of $n = 12, k = 3, p = 3$ by computing the requisite intermediates
\begin{align*}
T^1(9, 1, 1, 1) & = \{(9, 2, 1), (10, 1, 1)\}, \\
T^1(3, 3, 3, 3) & = \{(6, 3, 3)\}, \\
\pi_3 \zeta_3^{12} & = \tau(9, 2, 1) - \tau(10, 1, 1) + \tau(6, 3, 3).
\end{align*}
Since $m = 1 \le 1 = p - 2$, the above theorem then states that $(\tau(9, 2, 1) - \tau(10, 1, 1))$ and $\tau(6, 3, 3)$ are cocycles that form a basis for this subspace. The second case applies in essence when $\lambda$ corresponds to the base-$p$ representation of $n$; for instance, if $p = 3$ and $n = 8 = 2 \cdot 3^1 + 2 \cdot 3^0$, then $\lambda = (3, 3, 1, 1)$ is the smallest power-of-$3$ partition of $8$.  We can use the following information to form cocycle bases of dimensions $2$ and $3$:
\begin{align*}
T^1(3, 3, 1, 1) & = \{ (6, 1, 1), (4, 3, 1), (3, 3, 2) \} \\
\pi_3 \zeta^8_3 & = \tau(4, 3, 1) - \tau(6, 1, 1) - \tau(3, 3, 2) \\
T^2(3, 3, 1, 1) & = \{ (7, 1), (6, 2), (5, 3), (4, 4) \} \\
\pi_3 \zeta^8_2 & = \tau(7, 1) - \tau(6, 2) + \tau(5, 3) - \tau(4, 4)
\end{align*}
Theorem \ref{IntroCarryMinAreCocycles} then states that $\pi_3 \zeta^8_2$ and $\pi_3 \zeta^8_3$ span the spaces of characteristic $3$ cocycles of homogenous degree $8$ in dimensions $2$ and $3$ respectively.

To emphasize the interdimensional relationship that $T^m \lambda$ illuminates, we present the following stratifications of the first few cocycles of degrees $8$ and $12$ in characteristic $3$:

\begin{figure}[H]
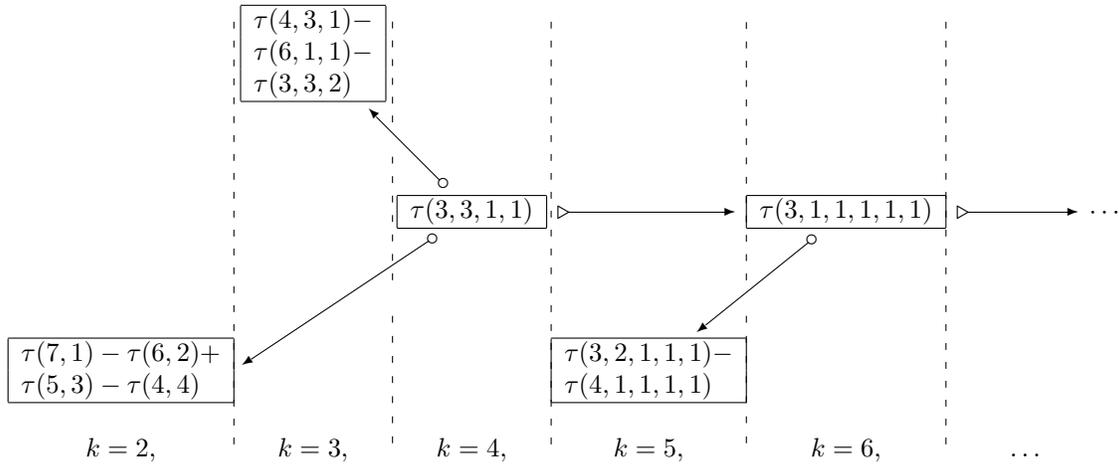

\begin{diagram}
                                    & \dAllDash & \drawbox{\tau(4, 3, 1) - \\
                                                           \tau(6, 1, 1) - \\
                                                           \tau(3, 3, 2)}   & \dAllDash      &                            & \dAllDash &                                 & \dAllDash & & \dAllDash\\
                                    &           &                           & \luGather      &                            &           &                                 &           & & \\
                                    &           &                           & \dAllDash      & \drawbox{\tau(3, 3, 1, 1)} &           & \rSplit                         &           & \drawbox{\tau(3, 1, 1, 1, 1, 1)} & & \rSplit & \cdots \\
                                    &           &                           & \ldGather(3,2) &                            &           &                                 & \ldGather & \\
\drawbox{\tau(7, 1) - \tau(6, 2) + \\
         \tau(5, 3) - \tau(4, 4)}   &           &                           & \dAllDash      &                            &           & \drawbox{\tau(3, 2, 1, 1, 1) - \\
                                                                                                                                                 \tau(4, 1, 1, 1, 1)}   & \dAllDash & \\
k = 2,                              &           & k = 3,                    &                & k = 4,                     &           & k = 5,                          &           & k = 6, & & \ldots
\end{diagram}
\caption{Cocycles over $\Z_3$ of homogenous degree $8$.}
\end{figure}
\begin{figure}[H]
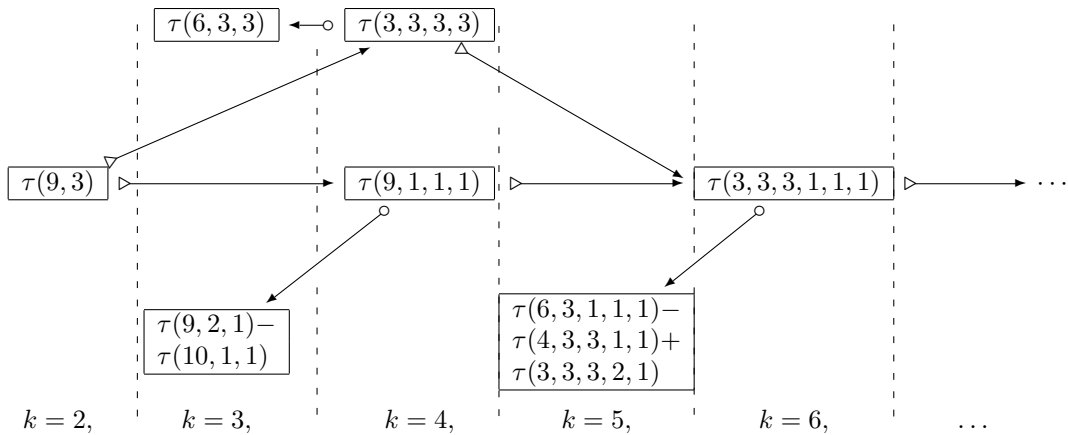

\begin{diagram}
                     & \dAllDash & \drawbox{\tau(6, 3, 3)}                     & \lGather                 & \drawbox{\tau(3, 3, 3, 3)} & \dAllDash       &       & \dAllDash & & \dAllDash \\
                     &           &                                             & \ruSplit(4, 2) \dAllDash &                            & \rdSplit(3, 2)  &       & \\
\drawbox{\tau(9, 3)} &           & \rSplit                                     &                          & \drawbox{\tau(9, 1, 1, 1)} & \dAllDash       &\rSplit& & \drawbox{\tau(3, 3, 3, 1, 1, 1)}  & & \rSplit & \cdots \\
                     &           &                                             & \ldGather(2, 2)          &                            &                 & & \ldGather(2, 2)    \\
                     &           & \drawbox{\tau(9, 2, 1) - \\ \tau(10, 1, 1)} & \dAllDash                &                            & & \drawbox{\tau(6, 3, 1, 1, 1) - \\
                                                                                                                                                    \tau(4, 3, 3, 1, 1) + \\
                                                                                                                                                    \tau(3, 3, 3, 2, 1)}     & \\
k = 2,               &           & k = 3,                                      &                          & k = 4,                     & & k = 5,                            & \uAllDash & k = 6, & & \ldots
\end{diagram}
\caption{Cocycles over $\Z_3$ of homogenous degree $12$.}
\end{figure} 
Here arrows with triangular tails denote cocycles discovered by splitting a power-of-$p$ partition, while those with circular tails denote cocycles discovered by our gathering operations and theorem \ref{IntroCarryMinAreCocycles}.

Under the conditions employed in \ref{IntroCarryMinAreCocycles}, $T^m \mu$ and $T^m \lambda$ are disjoint for distinct $\mu$ and $\lambda$, which gives us a way to count the cocycles of a particular degree and dimension.  It is well known that the coefficients $C^p_{nk}$ of the generating function \[\prod_{i = 0}^\infty \left( 1 - t x^{p^i} \right)^{-1} = \sum_{n, k} C^p_{nk} x^n t^k\] count the number of ways to write $n$ as a sum of $k$ powers of $p$.  Then, as a corollary,
\begin{ithm}\label{IntroGF}
The number of cocycles in degree $n$ and dimension $k$ is $C^p_{nk'}$, where $k'$ is the smallest number greater than $k$ such that $C^p_{nk'}$ is nonzero.
\end{ithm}

We begin the paper in \textsection\ref{Applications} by motivating the study of these $2$-cocycles and investigating loosely where they arise in other fields.  We then spend \textsection\ref{Preliminaries} introducing the notations and conventions used in our proofs, including multi-indices and the relevant cochain complex.  We break down the problem into smaller parts in \textsection\ref{BasicResults}, then work up to a version of \ref{IntroPOPCocycles} and the definition of $\zeta_k^n$.  In \textsection\ref{Gathering} we produce a variety of results about carry minimality and carry's behavior under our gathering and splitting operations, culminating in a proof of the first half of \ref{IntroCarryMinAreCocycles}.  We then spend \textsection\ref{IntegralProjection} on demonstrating that the cocycles we constructed in \textsection\ref{Gathering} form a basis, corresponding to the second half of \ref{IntroCarryMinAreCocycles}.  We wrap up the paper by giving a few corollaries of our classification in sections \textsection\ref{CountingAdditive2Cocycles} and \textsection\ref{TheGeneralizedLazardRing}, along with conjectures for higher cocycle conditions in \textsection\ref{ForHigherM}.

\section{Applications}\label{Applications}

\subsection{The Functor $\spec H_*BU \langle 2k \rangle$}\label{TheFunctorspecHBU2k}
Let $\{\bundleL_i\}_{i = 1}^{k+1}$ be $(k+1)$ copies of the line bundle over $\CP^\infty$ inherited from $\C^\infty \to \CP^\infty$ and denote the trivial line bundle on $\CP^\infty$ by $1$.  Vector bundles over a fixed space $X$ form a commutative semiring with addition given by direct sum and multiplication by tensor product.  Applying the Grothendieck construction to this semiring (which effectively adjoins formal additive inverses) produces what is called the ring of virtual bundles, denoted $K(X)$ or $K^0(X)$.  In this context we can consider the virtual bundle \[\xi_k = \prod_{i=1}^k (1 - \bundleL_i) = \sum_{I \subseteq \{1, \ldots, k\}} (-1)^{|I|} \left( \prod_{i \in I} \bundleL_i \right),\] taken over the product space $(\CP^\infty)^k$.

The virtual bundle $\xi_k$ is important because it illuminates a correspondence between ring maps \[H_* BU \langle 2k \rangle \to A\] and multiplicative $2$-cocycles in $k$ variables with coefficients in $A$ (here $X \langle n \rangle$ is the $(n-1)$-connected cover of $X$).  This arises in short because $\xi_k$ is a virtual bundle of rank zero, so there's a map $(\CP^\infty)^k \to BU$ classifying it.  Because $\xi_k$ has vanishing Chern classes $c_1, \ldots, c_k$ this classifying map lifts to $f: (\CP^\infty)^k \to BU \langle 2k \rangle$ in such a way that the following diagram commutes:
\begin{diagram}
\xi_k          &                               & BU \langle 2k \rangle \\
\dTo           & \ruTo^f                       & \dTo>{\hbox{connective fibration}}                  \\
(\CP^\infty)^k & \rTo_{\hbox{\hspace{3.5em}classifying map of $\xi_k$}} & BU
\end{diagram}

This induces a map in homology $H_* f: H_*(\CP^\infty)^k \to H_* BU \langle 2k \rangle,$ which by the universal coefficient theorem corresponds to an element \[f' \in H^*((\CP^\infty)^k; H_* BU \langle 2k \rangle).\]  An application of the K\"unneth formula and the calculation $H^* \CP^\infty = \Z \llbracket x \rrbracket$ gives that
\begin{align*}
H^*((\CP^\infty)^k ; H_* BU\langle 2k \rangle) & = \bigotimes_{i = 1}^k H^*(\CP^\infty; H_* BU\langle 2k \rangle) \\
& = (H_* BU\langle 2k \rangle)\llbracket x_1, \ldots, x_k\rrbracket,
\end{align*}
and so $f'$ can be viewed as a power series, an idea we further explore.  To start, there are also two standard classes of maps:
\begin{itemize}
{\item $\pi_i: (\CP^\infty)^{k+1} \to (\CP^\infty)^k$, given by dropping the $i$th copy of $\CP^\infty$.}
{\item $m_{ij}: (\CP^\infty)^{k+1} \to (\CP^\infty)^k$, given by applying the multiplication map $\CP^\infty \times \CP^\infty \to \CP^\infty$ to the $i$th and $j$th components, corresponding to tensor product of line bundles.}
\end{itemize}
We can then compute the pullback bundles along these maps:
\begin{align*}
\pi_s^* \xi_k & = \prod_{\substack{1 \le i \le k \\ i \ne s}}(1 - \bundleL_i), \\
m_{st}^* \xi_k & = (1 - \bundleL_s\bundleL_t) \cdot \prod_{\substack{1 \le i \le k \\ i \ne s, t}} (1 - \bundleL_i).
\end{align*}

Next, we make note of the following isomorphism (where $s \ne t$):
\begin{align*}
(m_{st}^* - \pi_s^* - \pi_t^*) \xi_k & = (1 - \bundleL_s \bundleL_t) \cdot \prod_{\substack{1 \le i \le k + 1 \\ i \ne s, t}} (1 - \bundleL_i) - \prod_{\substack{1 \le i \le k + 1 \\ i \ne s}} (1 - \bundleL_i) - \prod_{\substack{1 \le i \le k + 1 \\ i \ne t}} (1 - \bundleL_i) \\
& = ((1 - \bundleL_s \bundleL_t) - (1 - \bundleL_s) - (1 - \bundleL_t)) \prod_{\substack{1 \le i \le k + 1 \\ i \ne s, t}} (1 - \bundleL_i) \\
& = \prod_{1 \le i \le k+1} (1 - \bundleL_i) = \xi_{k+1}.
\end{align*}
In particular, this means the following class of isomorphisms hold for all choices of $s \ne t, s' \ne t'$: \[(m_{st}^* - \pi_s^* - \pi_t^*)\xi_k \cong (m_{s't'}^* - \pi_{s'}^* - \pi_{t'}^*)\xi_k.\]  Selecting $s = 1$, $t = s' = 2$, and $t' = 3$ gives the following identity in terms of our power series $f'$: \[\frac{f'(x_1, \ldots, x_k)}{f'(x_0 + x_1, x_2, \ldots, x_k)} \cdot \frac{f'(x_0, x_1 + x_2, x_3, \ldots, x_k)}{f'(x_0, x_1, x_3, \ldots, x_k)} = 1.\]  We call this the multiplicative $2$-cocycle condition in $k$ variables.

In addition, there are maps $t_{ij}: (\CP^\infty)^k \to (\CP^\infty)^k$ that act by transposing the $i$th and $j$th coordinates, and the isomorphism of virtual bundles $t_{ij}^* \xi_k \cong \xi_k$ means that $f'$ is symmetric as a power series.  The map $i_s: (\CP^\infty)^{k-1} \into (\CP^\infty)^k$ that includes away from the $s$th factor can be composed with $\pi_s$ to give a pullback bundle \[(i_s\pi_s)^*\xi_k = (1 - 1)\prod_{\substack{1 \le i \le k \\ i \ne s}}(1 - \bundleL_i) = 0,\] which in turn forces $f'$ to be a rigid power series (i.e., $f'(\ldots, 0, \ldots) = 1$).  This guarantees the existence of the power series $(f')^{-1}$ used above.

For $k \le 3$, theorems due to Ando, Hopkins, and Strickland state that multiplicative 2-cocycles over an arbitrary ring $A$ are selected by this universal multiplicative $2$-cocycle $f'$ via some ring map $H_* BU\langle 2k \rangle \to A$ and the induced map $(H_* BU\langle 2k \rangle)\llbracket x_1, \ldots, x_k \rrbracket \to A \llbracket x_1, \ldots, x_k \rrbracket$.  In addition, the action of a ring map $H_* BU \langle 2k \rangle \to A$ is determined by the image of $f'$, effectively giving a polite description of $\spec H_* BU\langle 2k \rangle$.  For $k > 3$, the state of this correspondence is not known.

It is easy to check that these power series must be of the form $1 + g + \hbox{higher order terms}$, where $g$ is a $k$-variable additive 2-cocycle as described in \textsection\ref{Overview}.  Classifying the additive cocycles restricts where we should look to extend to multiplicative cocycles; in an algebraic geometric sense, the additive 2-cocycles surject onto the tangent space of multiplicative 2-cocycles.  This is then the first step in exploring the topological relationship described above.

\subsection{Formal Group Laws}\label{FormalGroupLaws}

Let $E^*$ be a multiplicative cohomology theory and $E = E^*(\mathrm{pt})$ be its coefficient ring.  $E^*$ is said to be complex orientable when it admits a notion of Chern classes; given a vector bundle $\xi / X$, the Chern classes associated to $\xi$ under $E^*$ are a sequence of elements $c^E_i(\xi) \in E^{2i}(X)$ satisfying the following properties:
\begin{itemize}
\item{{\bf Naturality:}  Given a map $f: Y \to X$, we have $f^* (c^E_i \xi) = c^E_i(f^* \xi)$, where the first use of $f^*$ denotes the induced map in cohomology and the second use denotes the pullback bundle construction.}
\item{{\bf Additivity:} For vector bundles $\xi / X$ and $\eta / X$, we have \[c^E_n(\xi \oplus \eta) = \sum_{i = 0}^n c^E_i(\xi) c^E_{n - i}(\eta).\]  If we write the ``total Chern class'' as the formal power series $c^E(\xi) = \sum_i c^E_i(\xi)$, this can be expressed as $c^E(\xi \oplus \eta) = c^E(\xi) c^E(\eta)$.}
\item{{\bf Normalization:} We require the cohomology theory to provide an element $x \in E^2 \CP^\infty$ such that $E^* \CP^\infty = E \llbracket x \rrbracket$ and $E^* \CP^k = E \llbracket x \rrbracket / \langle x \rangle^{k+1}$, and we require the first Chern class to behave as $c^E_1(\bundleL) = x$.  As in \textsection\ref{TheFunctorspecHBU2k}, $\bundleL$ is the line bundle over $\CP^\infty$ inherited from the quotient map $\C^\infty \to \CP^\infty$.}
\end{itemize}

Let $f: \CP^\infty \times \CP^\infty \to \CP^\infty$ be the map classifying the line bundle ($\bundleL \otimes \bundleL) / (\CP^\infty \times \CP^\infty)$.  $f$ induces a map in cohomology of the form \[f^*: E^*(\CP^\infty) \to E^*(\CP^\infty \times \CP^\infty) = E \llbracket x, y \rrbracket.\]  Applying the axioms above, we find that the first Chern class of the product bundle takes the form \[c^E_1(\bundleL \otimes \bundleL) = c^E_1(f^* \bundleL) = f^*(x) = F(x, y)\] for some bivariate power series $F$.  Various properties of the tensor product of line bundles force the following three properties upon $F$:
\begin{align*}
\xi \otimes \eta \cong \eta \otimes \xi & \implies F(x, y) = F(y, x), \\
\xi \otimes 1 \cong 1 \otimes \xi \cong \xi & \implies F(x, 0) = F(0, x) = x, \\
(\xi \otimes \eta) \otimes \nu \cong \xi \otimes (\eta \otimes \nu) & \implies F(F(x, y), z) = F(x, F(y, z)).
\end{align*}
Any $F$ satisfying these three properties we call a (commutative, one-dimensional) ``formal group law.''  As examples, the FGL associated in this way to ordinary cohomology theory is $\G_a(x, y) = x + y$, the FGL associated to complex K-theory is $\G_m(x, y) = x + y + xy$, and the FGL associated to complex cobordism is $U(x, y) = \sum_i a_i \cdot f_i(x, y)$, where $f_i$ are Lazard's cocycles from \textsection\ref{Overview} and the $a_i$ are the adjoined elements in $MU^*(\mathrm{pt}) \cong \Z[a_1, a_2, \ldots]$.

The cocycle condition also manifests in this context.  Given a ring $A$, we define an $n$-bud (sometimes called an $n$-chunk) to be a polynomial $f \in A \llbracket x, y \rrbracket / (x, y)^{n+1}$ also satisfying the above three conditions, a sort of truncated formal group law.  Now let $f' \in A \llbracket x, y \rrbracket / (x, y)^{n+2}$ be an $(n+1)$-bud extending $f$ (i.e., $f' = f \mod (x, y)^{n+1}$).  We seek to classify the polynomials $g \in A \llbracket x, y \rrbracket / (x, y)^{n+2}$ such that $f' + g$ is again an $(n+1)$-bud extension of $f$ (the motivation here is that two $(n+1)$-buds extending $f$ will differ by such a $g$).

To begin, $g$ must be of homogenous degree $(n+1)$ since it must vanish under the map \[A \llbracket x, y \rrbracket / (x, y)^{n+1} \onto A \llbracket x, y \rrbracket / (x, y)^n.\]  In addition, since $f' + g$ is an $(n+1)$-bud, their sum must satisfy the three FGL axioms, and in particular \[(f'+g)((f'+g)(x, y), z) = (f'+g)(x, (f'+g)(y, z)).\]  After noting that $\G_a(x, y)$ is simple and trivially both an $n$-bud and an $(n+1)$-bud, we set $f = f' = \G_a$ for ease of computation.  Keeping careful track of truncation degree we see:
\begin{align*}
(f'+g)((f'+g)(x, y), z) & = (f'+g)(x, y) + z + g((f'+g)(x, y), z) \\
& = x + y + z + g(x, y) + g(x + y + g(x, y), z) \\
& = x + y + z + g(x, y) + g(x + y, z), \\ \\
(f'+g)(x, (f'+g)(y, z)) & = x + (f'+g)(y, z) + g(x, (f'+g)(y, z)) \\
& = x + y + z + g(y, z) + g(x, y + z + g(y, z)) \\
& = x + y + z + g(y, z) + g(x, y + z).
\end{align*}
Equating these two expansions forces the relation \[g(x, y) - g(z + x, y) + g(z, x + y) - g(z, x) = 0,\]  and $g$ is then said to be a symmetric additive $2$-cocycle.


\subsection{Split Extensions and Higher Cubical Structures}

In this section, all groups are assumed abelian.  Let $A$ and $C$ be groups.  A group $B$ with homomorphisms $\pi, i$ and set map $s$ is said to be a ``split extension'' of $C$ by $A$ if the sequence \[0 \to A \stackrel{i}{\into} B \stackrel{\pi}{\onto} C \to 0\] is exact and $s$ satisfies both $\pi \circ s = id_C$ and $s(0_C) = 0_B.$  It is fairly obvious that $B \cong A \times C$ as sets; we can explicitly construct the two halves of the set isomorphism: \[\gamma: (a, c) \mapsto a + s c,\] \[\gamma^{-1}: b \mapsto (b - s \pi b, \pi b).\]  For brevity we have identified $A$ with its embedding in $B$, and we will shorthand the $A$ component of $\gamma^{-1}$ as $\alpha: b \mapsto b - s \pi b$.

We can then investigate the group structure induced on $A \times C$:
\begin{align*}
\gamma \gamma^{-1} b_1 + \gamma \gamma^{-1} b_2 & = \gamma \gamma^{-1}(b_1 + b_2) \\
\gamma(\alpha b_1, \pi b_1) + \gamma(\alpha b_2, \pi b_2) & = \gamma(\alpha(b_1 + b_2), \pi(b_1 + b_2)) \\
& = \gamma(b_1 + b_2 - s \pi (b_1 + b_2), \pi(b_1 + b_2)) \\
& = \gamma(\alpha b_1 + s \pi b_1 + \alpha b_2 + s \pi b_2 - s \pi (b_1 + b_2), \pi(b_1 + b_2)) \\
& = \gamma(\alpha b_1 + \alpha b_2 + (s \pi b_1 + s \pi b_2 - s \pi (b_1 + b_2)), \pi(b_1 + b_2))
\end{align*}
If we rename these elements as $(a_1, c_1) = \gamma^{-1} b_1, (a_2, c_2) = \gamma^{-1} b_2$, then the last equality takes the form \[(a_1, c_1) +_B (a_2, c_2) = (a_1 + a_2 + f_s(c_1, c_2), c_1 + c_2),\] \[f_s(c_1, c_2) = s c_1 + s c_2 - s(c_1 + c_2).\]  As in the formal group law computation in \textsection\ref{FormalGroupLaws}, checking associativity and commutativity of $+_B$ forces the $2$-cocycle and symmetry conditions on $f$.

It is important to note that we've converted a section $s: C \to B$ into a map $f: C^2 \to A$.  If we define a map of split extensions to be a map of short exact sequences such that the sections also commute, then it is easy to check that two split extensions are isomorphic if and only if their associated cocycles differ by a coboundary (i.e., for $f$ and $f'$ the associated $2$-cocycles, we can find $g: C \to A$ such that $(f - f')(c_1, c_2) = g(c_1 + c_2) - g(c_1) - g(c_2)$).  Again, this associated cochain complex in no way involves the extension $B$ groups themselves.

This can all be restated by letting $B$ be an $A$-torsor over $C$ (we'll change notation to $\torsorB$ to emphasize the change of setting); the section $s$ then ``trivializes'' the torsor.  In this light, we can use various standard constructions to build new torsors out of these old ones, including:

\begin{itemize}
{\item {\bf Pullback:} Given a $G$-torsor $\torsorB \to Y$ and a set map $f: X \to Y$, we can construct the pullback torsor $f^* \torsorB \to X$ whose fibers are given by $(f^* \torsorB)_x = \torsorB_{f(x)}$.}
{\item {\bf Dual:} Given a $G$-torsor $\torsorB \to X$ we can construct a torsor $\torsorB^{-1} \to X$, called the dual of $\torsorB$, whose fiber over $x \in X$ is given by $G$-equivariant maps $\torsorB_x \to G$.}
{\item {\bf Tensor product:} Given two $G$-torsors $\torsorA, \torsorB \to X$ we can construct a torsor $\torsorA \otimes \torsorB \to X$ whose fibers are given by $(\torsorA \otimes \torsorB)_x = \torsorA_x \otimes_G \torsorB_x$.  Denoting the trivial torsor $G \times X \to X$ by $1$, the notation for the dual is then motivated by the relations $\torsorA \otimes \torsorB \cong \torsorB \otimes \torsorA$, $\torsorA \otimes (\torsorB \otimes \torsorC) \cong (\torsorA \otimes \torsorB) \otimes \torsorC$, $\torsorA \otimes 1 \cong \torsorA$, and $\torsorA \otimes \torsorA^{-1} \cong 1$.}
\end{itemize}

In the case that $X$ is a group, we have a number of projection and multiplication maps $X^{m+1} \to X^m$ analogous to those given in \textsection\ref{TheFunctorspecHBU2k}.  We use this information to define a symmetric biextension of $C$ by $A$ to be an $A$-torsor $\torsorB \to C^2$, along with a section $t$ of the torsor $\torsorB^{-1} \otimes \tau^* \torsorB \to C^2$, $\tau: C^2 \to C^2$ the flip map, and sections $s_{ij}$ of the following family torsors for each $1 \le i \le 2$, $j = i + 1$: \[\chi_{ij}\torsorB = \frac{m_{ij}^* \torsorB}{\pi_i^* \torsorB \otimes \pi_j^* \torsorB} \to C^3.\]  Because these sections trivialize these torsors, we can translate these into the fiber relations $\torsorB_{c + c', d} \cong \torsorB_{c, d} \otimes \torsorB_{c', d}$, $\torsorB_{c, d + d'} \cong \torsorB_{c, d} \otimes \torsorB_{c, d'}$, and $\torsorB_{c, d} \cong \torsorB_{d, c}$.  These fiber relations express a sort of partial group law defined on $\torsorB$ whenever the two operands share a $C$-component.  Such a $\torsorB$ equipped with a section $s: C^2 \to \torsorB$ is called a split symmetric biextension, and as in the split extension case we can explicitly write out the (partial) group laws as \[(a, x, y) + (b, x', y) = (a + b + f(y)(x, x'), x + x', y),\] \[(a, x, y) + (b, x, y') = (a + b + f(x)(y, y'), x, y + y'),\] where each $f(x)(-,-)$ is a symmetric 2-cocycle.


Now, given a torsor $\torsorB \to C$, we can construct the two torsors $\Lambda \torsorB \to C^2$ and $\Theta \torsorB \to C^3$, called the first and second differences of $\torsorB$ respectively, whose fibers are given by the formulas \[(\Lambda \torsorB)_{x, y} = \frac{\torsorB_{x + y}}{\torsorB_x \otimes \torsorB_y}, (\Theta \torsorB)_{x, y, z} = \frac{\torsorB_{x + y + z} \otimes \torsorB_x \otimes \torsorB_y \otimes \torsorB_z}{\torsorB_{x + y} \otimes \torsorB_{x + z} \otimes \torsorB_{y + z}}.\]  A torsor is said to be rigid when we equip it with a section of the fiber $\torsorB_{0}$; a section of $\Theta \torsorB$ then automatically gives a rigidification of $\torsorB$, $\Lambda \torsorB$, and $\Theta \torsorB$.  The section of $\Theta \torsorB$ itself is said to be rigid when the rigidification section agrees with the induced sections of $i_{s}^* (\Theta \torsorB) \cong (\pi \circ 0)^* (\Theta \torsorB)$, where $i_s: C^2 \to C^3$ includes away from the $s$th component.  A rigid section of $\Theta \torsorB$ corresponds to a kind of symmetric biextension structure on $\Lambda \torsorB$ called a cubical structure.  We pick first fiber relation given in the previous paragraph to derive as an example:
\begin{align*}
\Theta \torsorB = \frac{\torsorB_{x + y + z} \otimes \torsorB_x \otimes \torsorB_y \otimes \torsorB_z}{\torsorB_{x + y} \otimes \torsorB_{x + z} \otimes \torsorB_{y + z}} & \cong 1 \\
\frac{\torsorB_{x + z}}{\torsorB_x \otimes \torsorB_z} \otimes \frac{\torsorB_{y + z}}{\torsorB_y \otimes \torsorB_z} & \cong \frac{\torsorB_{x + y + z}}{\torsorB_{x + y} \otimes \torsorB_z} \\
(\Lambda \torsorB)_{x, z} \otimes (\Lambda \torsorB)_{y, z} & \cong (\Lambda \torsorB)_{x + y, z}.
\end{align*}
Thus, since giving a section of $\Theta \torsorB$ trivializes it, we get a biextension structure on $\Lambda \torsorB$ because of it, and the biextension structure is automatically symmetric by definition of $\Lambda \torsorB$.  In fact, because the maps $(\Lambda \torsorB)_{x, z} \otimes (\Lambda \torsorB)_{y, z} \to (\Lambda \torsorB)_{x + y, z}$ and $(\Lambda \torsorB)_{x, y} \otimes (\Lambda \torsorB)_{z, y} \to (\Lambda \torsorB)_{x + z, y}$ are both determined by the same section of $(\Theta \torsorB)_{x, y, z}$ (``same'' in the sense that the section is rigid, and so it won't matter which we choose), we have that the two evaluations of $f$ in the following two calculations are equal:
\begin{align*}
(g, x, z) +_L (h, y, z) & = (g + h + f(z)(x, y), x + y, z) \\
(g, x, y) +_L (h, z, y) & = (g + h + f(y)(x, z), x + z, y),
\end{align*}
where $+_L$ denotes the action of the isomorphisms $\torsorB_{x, z} \otimes \torsorB_{y, z} \to \torsorB_{x + y, z}$ given by the biextension structure.  Similar equalities occur for other permutations of $x$, $y$, and $z$, resulting in symmetry of $f$ as a function $C^3 \to G$.  This material has all been examined in detail before; see for instance \cite{Bre83} for a thorough treatment of cubical structures in general and \cite{AS01} for their application as in \textsection\ref{TheFunctorspecHBU2k}.

We can use a variation of this construction to form $m$-variable $2$-cocycles $f: C^m \to A$.  Given an $A$-torsor $\torsorB \to C$, let $\Theta^m \torsorB \to C^m$ be defined by the formula \[(\Theta^m \torsorB)_{\bf x} = \bigotimes_{\substack{I \subseteq \{1, \ldots, m\} \\ I \ne \emptyset}} \left( \torsorB_{\sum_{i \in I} {\bf x}_i} \right)^{(-1)^{|I|}}.\]  It's worth noting the following correspondences:
\begin{align*}
\Theta^0 \torsorB & = 1, \\
\Theta^1 \torsorB & = \torsorB, \\
\Theta^2 \torsorB & = \Lambda \torsorB, \\
\Theta^3 \torsorB & = \Theta \torsorB.
\end{align*}
Generalizing the previous definitions in the obvious way, an $m$-extension is a $\torsorB \to C^m$ with sections of $\chi_{ij} \torsorB \to X^{m+1}$ for $1 \le i \le m$, $j = i+1$, and a symmetric $m$-extension is an $m$-extension where $\torsorB_{\bf x} \cong \torsorB_{\sigma {\bf x}}$ for every $\sigma \in S_m$.  Then, a section $s$ of $\Theta^{m+1} \torsorB \to C^{m+1}$ (a sort of higher cubical structure) satisfying $\pi_A s({\bf x}) = \pi_A s(\sigma {\bf x})$ induces a symmetric $m$-extension structure on $\Theta^m \torsorB \to C^m$ in a manner identical to the biextension case.  Again as in the previous cases, the symmetric $m$-extension structure gives rise to a function $f: C^{m-2} \to (C^2 \to A)$ which parameterizes a family of symmetric $2$-cocycles, and as in the biextension case because the same fiber section of $\Theta^{m+1} \torsorB$ determines the action of both $f(\bf{x})$ (here interpreted as a function $f: C^m \to A$) and $f(\sigma \bf{x})$, we find that $f({\bf x}) = f(\sigma {\bf x})$.

We can recast this again, this time in the light of affine schemes: to give a split extension of the group scheme $Z$ by the group scheme $X$ is to give a split extension $Y(R)$ of the groups $Z(R)$ by $X(R)$ naturally in $R$.  This is to say that for every ring map $f: R \to S$ we should have the following corresponding commutative diagram:
\begin{diagram}
0 & \rTo & X(R)        & \rTo & Y(R)        & \rTo & Z(R)        & \rTo & 0 \\
  &      & \dTo>{X(f)} &      & \dTo>{Y(f)} &      & \dTo>{Z(f)} &      &   \\
0 & \rTo & X(S)        & \rTo & Y(S)        & \rTo & Z(S)        & \rTo & 0
\end{diagram}
In addition, we require that $Y(f) \circ s(R) = s(S) \circ Z(f)$, where $s(R): Z(R) \to Y(R)$ is the section associated to the split extension $Y(R)$ of $Z(R)$ by $X(R)$.

Let $\G_a$ denote the functor that sends a $k$-algebra $R$ to its underlying additive group $R_+$, representable by $k[x]$.  If we fix a split extension $Y$ of $\G_a$ by $\G_a$ and pick a $k$-algebra $R$, then the split extension $Y(R)$ associated to $R$ is set isomorphic to $R_+ \times R_+$, and the multiplication map $Y(R) \times Y(R) \to Y(R)$ then corresponds to a map $(R_+ \times R_+)^2 \to R_+ \times R_+$.  We have seen already that the multiplication in $Y(R)$ is determined by its action on elements with zero left-component, say $(0, r)$ and $(0, s)$.  These elements are, by construction of $\G_a(R) = X(R) = Z(R)$, selected by the map \[f: k[a] \otimes_k k[b] = k[a, b] \to (R_+)^2,\] \[ f: a \mapsto r, f: b \mapsto s. \]  By naturality of the scheme assignment,
\begin{align*}
(0, r) +_{Y(R)} (0, s) & = (f(0), f(a)) +_{Y(R)} (f(0), f(b)) \\
& = f \left( (0, a) +_{Y(k[a, b])} (0, b) \right) \\
& = f \left( (g(a, b), a + b) \right),
\end{align*}
where $g$ is the symmetric $2$-cocycle corresponding to the split extension.  Most importantly, $g$ is a map with target $k[a, b]$, and so $g(a, b)$ will be a polynomial over $k$ that universally determines the action of the split scheme extension.  $g$ is easily seen to be symmetric and to satisfy the $2$-cocycle condition.  This same construction can be made for split multiextensions of $\G_a$ by $\G_a$, where the $k$-variable symmetric $2$-cocycle again has a polynomial representation.


\section{Characterization of Additive Cocycles}

\subsection{Preliminaries}\label{Preliminaries}

We first introduce the central constructs and notations we will use throughout the paper, most importantly that of multi-indices and number theoretic functions on them, in particular the notion of carry-count.

\begin{defn}\label{MultiindexDefns} 
A \emph{multi-index} of weight $n$ and length $k$ is a $k$-tuple of elements of $\N_0=\N\cup\{ 0\}$ of the form $\lambda = (\lambda_1, \lambda_2, \ldots, \lambda_k)$ that satisfies $\sum_i \lambda_i = n$.  We further say that $\lambda$ is a \emph{power-of-$p$ multi-index} when there exist $a_i \in \N$ such that $\lambda_i = p^{a_i}$ for all $i$.  We denote the length as $\length{\lambda} = k$ and the weight as $\weight{\lambda} = n$.

We define the following operations over multi-indices:

\begin{itemize}
{\item {\bf Exponentiation:} ${\bf x}^\lambda = x_1^{\lambda_1}x_2^{\lambda_2}\cdots x_k^{\lambda_k}$, where ${\bf x} = (x_1, x_2, \ldots, x_k)$.}
{\item {\bf Permutation:} $\sigma \lambda = (\lambda_{(\sigma 1)}, \lambda_{(\sigma 2)}, \ldots, \lambda_{(\sigma k)})$, for $\sigma \in S_k$.}
{\item {\bf Membership:} We write $a \in \lambda$ when there is some $i$ for which $a = \lambda_i$.}
{\item {\bf Concatenation:} $\lambda \cup \mu = (\lambda_1, \ldots, \lambda_i, \mu_1, \ldots, \mu_j)$.}
{\item {\bf Difference:} $\lambda \setminus \mu = \lambda'$ is defined to be the unique unordered multi-index such that $\mu \cup \lambda' = \lambda$ (again up to reordering).  For example, $(2, 2, 1) \setminus (2, 1) = (2)$.}
{\item {\bf Ring extension:} $A[{\bf x}] = A[x_1, \ldots, x_k]$ for ${\bf x} = (x_1, \ldots, x_k)$.}
{\item {\bf Map to monomials:} We define $\tau(\lambda)$ to be the polynomial $\sum_{\sigma \in S_k} {\bf x}^{\sigma \lambda}$ once divided by the gcd of the coefficients.  For example, we provide these expansions:
\begin{align*}
\tau(2, 1, 1) & = x_1^2 x_2 x_3 + x_1 x_2^2 x_3 + x_1 x_2 x_3^2 \\
\tau(2, 2, 2) & = x_1^2 x_2^2 x_3^2 \\
\tau(1, 2, 3) & = x_1 x_2^2 x_3^3 + x_1 x_2^3 x_3^2 + x_1^2 x_2 x_3^3 + x_1^3 x_2 x_3^2 + x_1^2 x_2^3 x_3 + x_1^3 x_2^2 x_3
\end{align*}}
{\item {\bf Factorial:} We define $\lambda ! = \prod_i \lambda_i!$.}
\end{itemize}

In addition, there are a handful of useful number-theoretic constructs that can be formulated in terms of multi-indices:
\begin{itemize}
{\item {\bf Partitions:} When all entries of a multi-index of weight $n$ are positive and listed in descending order, it is called a \emph{partition of n} and is denoted $\lambda \vdash n$.}
{\item {\bf Multinomial coefficients:} For a multi-index $\lambda$, let ${\weight{\lambda} \choose \lambda}$ denote the integer $(\weight{\lambda})!(\lambda !)^{-1}$.  Note that \[{n+m \choose (m, n)} = {n + m \choose m} = {n + m \choose n}\] corresponds with the usual binomial coefficients.}
{\item {\bf Carry count:} The number of times one carries when calculating the base $p$ sum $\sum_i \lambda_i$ is denoted $\alpha_p(\lambda)$.  It is well known that this can be formalized as the number of times ${\weight{\lambda} \choose \lambda}$ is divisible by $p$.  This is a straightforward generalization of a result due to Kummer, originally found in \cite{Kum852}.  A particularly useful property is that for two multi-indices $\lambda, \mu$ we have $\alpha_p(\lambda \cup \mu) = \alpha_p((\weight{\lambda}) \cup \mu) + \alpha_p(\lambda)$, corresponding to associativity of addition.}
{\item {\bf Digital sum:} The digital sum of a number $n$ in base $p$ is denoted $\sigma_p(n)$.  Explicitly, if $n = \sum_{i=0}^\infty a_i p^i$ for $0 \le a_i < p$, then $\sigma_p(n) = \sum_{i=0}^\infty a_i$.}
{\item {\bf Base-$p$ representation:} Given $n \in \N$, let $\rho_p(n)$ be the power-of-$p$ multi-index such that $\weight{\rho_p(n)} = n$ and $\rho_p(n)$ has minimal length (i.e., $\sigma_p(n) = \length{\rho_p(n)}$).  For example, we can compute $\rho_3(16) = (9, 3, 3, 1)$.  $(9, 3, 1, 1, 1, 1)$ is another power-of-$3$ multi-index with weight $16$, but it does not have minimal length.  This corresponds in an obvious way to the base-$p$ expansion of $n$ given by $n = \sum_{i=0}^\infty a_i p^i$ for $0 \le a_i < p$, where $p^i$ appears $a_i$ many times in $\rho_p(n)$.}
\end{itemize}
\end{defn} 

\begin{defn}\label{CarryMinimalDefn} 
We say that a partition $\lambda$ of weight $n$ and length $k$ has \emph{carry-minimal sum in base $p$} or is \emph{$p$-carry minimal} when $\alpha_p(\lambda)$ is minimal in the sense of $\alpha_p(\lambda) = \min \{\alpha_p(\lambda ') \mid \lambda' \vdash n, \length{\lambda'} = k\}$.  For example, $\alpha_3(9,2,1)=1$, and is $3$-carry minimal.  $\alpha_3(8,3,1) = 2$, and so because $(8,3,1)$ is of the same weight and length as $(9,2,1)$, it is not $3$-carry minimal.
\end{defn} 

\begin{defn}\label{ProjDefn} 
Throughout this paper, we will use \emph{ring} to mean commutative ring with unit.  Given a ring $A$ and an ideal $I \subseteq A$ we will use $\pi_I: A \to A / I$ to denote the natural homomorphism with kernel $I$.  In the event $I=\langle a\rangle,$ $a\in A$, we denote $\pi_I$ by $\pi_a$.
\end{defn} 

\begin{defn}\label{SymmDefn} 
We say an $k$-variable polynomial is \emph{symmetric} if $f({\bf x}) = f(\sigma {\bf x})$ for all $\sigma \in S_k$.
\end{defn} 

\begin{rem}\label{Gradation}
The $A$-algebra of symmetric multivariate polynomials has two natural gradations, one corresponding to degree and one corresponding to number of variables.
\end{rem}

\begin{defn}\label{SymmBasisDefn} 
$\tau$ surjects onto a basis for symmetric polynomials.  When restricted to $k$-variable polynomials of homogenous degree $n$, we call it the \emph{monomial symmetric basis on $k$ variables}, and denote it as $B_k^n$.
\end{defn} 

\begin{defn}\label{ACocycleDefn} 
The $m$-\emph{coboundary map}, denoted $\coboundarymap_m$, is a map of modules that operates on polynomials of $k \ge m$ variables and is defined by
\begin{align*}
\coboundarymap_m(f) & = f(x_1, \ldots, x_k) \\
& + \sum_{i=1}^m (-1)^i f(x_0, x_1, \ldots, x_{i-2}, x_{i-1} + x_i, x_{i+1}, \ldots, x_k) \\
& + (-1)^{m+1} f(x_0, x_1, \ldots, x_{m-1}, x_{m+1}, \ldots, x_k).
\end{align*}

It is easy to see that $\coboundarymap_m$ sends polynomials of homogenous degree $n$ in $k$ variables to polynomials of homogenous degree $n$ in $(k+1)$ variables.  In addition, $\coboundarymap_m \circ \coboundarymap_{m-1} = 0,$ or $\coboundarymap$ is a differential.  We define the \emph{($m$-)cocycle condition} as applied to a polynomial $f$ to mean $\coboundarymap_m f = 0$, and say that ``$f$ satisfies the ($m$-)cocycle condition'' or ``$f$ is an ($m$-)cocycle''.\end{defn} 

We also define a number of one-time use functions.  We will reuse the symbol $\theta$ for all of them to save naming clutter.  Which $\theta$ we intend will be clear, since a definition will be given in the theorem statement.


\subsection{Basic Results}\label{BasicResults}

Beginning at this point, we strongly suggest that the reader frequently refer to appendix \ref{TablesOfModulesAdditive2Cocycles}, where we list bases for $\ker \coboundarymap_2$ restricted to particular degrees and dimensions with coefficients in $\Z_2$, $\Z_3$, and $\Z_5$.  The structure of the data guides the structure of the proofs to follow, and to reinforce this we will provide some examples inlined with the body of the text.

Given that $\coboundarymap_m$ is a graded map of modules, we then seek to further decompose the problem into more workable pieces.  There is a basis for the module of all polynomials $A[{\bf x}]$ given by $\{{\bf x}^\lambda\}_\lambda$ for all multi-indices $\lambda$.  We would like to have the additional ability to consider our monomial symmetric basis elements one monomial at a time, but we run into the complication that there exist monomials shared between different monomial symmetric basis elements depending upon the dimension of the grading -- for instance, $\tau(1, 1, 0)$ and $\tau(1, 1)$ share the term $x_1 x_2$.  To eliminate this problem, we show in general that symmetrized monomials with terms not mixed in every variable cannot participate in polynomials in $\ker \coboundarymap_m$.

\begin{lem}\label{NoInitialZeroes} 
If $\lambda$ is of the form $(0, \lambda_2, \lambda_3, \ldots, \lambda_k)$, then ${\bf x}^\lambda$ cannot contribute to a linear combination of monomials in the kernel of $\coboundarymap$.
\end{lem}
\begin{proof}
The application $\coboundarymap_m({\bf x}^\lambda)$ yields the following sum:
\begin{align*}
\coboundarymap_m({\bf x}^\lambda) & = x_2^{\lambda_2} x_3^{\lambda_3} \cdots x_k^{\lambda_k} - x_2^{\lambda_2} x_3^{\lambda_3} \cdots x_k^{\lambda_k} + (x_1 + x_2)^{\lambda_2} x_3^{\lambda_3} \cdots x_k^{\lambda_k} \\
& + \sum_{i=3}^m (-1)^i x_1^{\lambda_2} \cdots x_{i-2}^{\lambda_{i-1}}(x_{i-1} + x_i)^{\lambda_i}x_{i+1}^{\lambda_{i+1}} \cdots x_k^{\lambda_k} \\
& + (-1)^{m+1} x_1^{\lambda_2} \cdots x_{m-1}^{\lambda_m} x_{m+1}^{\lambda_{m+1}} \cdots x_k^{\lambda_k}.
\end{align*}
Ignoring the monomials with terms mixed in $x_1$ and $x_2$ (equivalently, working modulo the ideal $\langle x_0 \cdots x_k \rangle$), we see that we have a residual term of $x_2^{\lambda_2} x_3^{\lambda_3} \cdots x_k^{\lambda_k}$.  Any other choice of $\lambda$ will yield summands distinct from this monomial, therefore ${\bf x}^\lambda$'s image cannot be completely cancelled by any other monomial's image under $\coboundarymap_m$.  Thus no linear combination of monomials containing ${\bf x}^\lambda$ can lie in the kernel of $\coboundarymap_m$.
\end{proof} 

\begin{cor}\label{NoSymmetricZeroes} 
If $0 \in \lambda$, then $\tau \lambda$ cannot contribute to any symmetric cocycle.
\end{cor}
\begin{proof}
Select a $\sigma \in S_k$ such that $\sigma \lambda$ is of the form $(0, \lambda_2, \ldots, \lambda_k)$ and apply \ref{NoInitialZeroes}.
\end{proof} 

We can then restrict our attention to multi-indices $\lambda$ that satisfy $0 \not\in \lambda$.  Since the presence of a zero was the only thing that prevented the entire sum from telescoping in \ref{NoInitialZeroes}, we also note that any unmixed terms in the na\"ive expansion of $\coboundarymap_m \tau \lambda$ will vanish.

\begin{lem}\label{OnlyMixedTerms} 
Let $\lambda$ be a multi-index with $0 \not\in \lambda$, $\length{\lambda} = k$.  Then $\coboundarymap_m({\bf x}^\lambda)$ will contain only monomials mixed in all of $x_0, \ldots, x_k$.
\end{lem}
\begin{proof}
Again working modulo the ideal $\langle x_0 \cdots x_k \rangle$, $\coboundarymap_m({\bf x}^\lambda)$ can be rewritten as:
\begin{align*}
x_1^{\lambda_1} \cdots x_k^{\lambda_k} & +
\sum_{i=1}^m (-1)^i x_0^{\lambda_1} \cdots x_{i-1}^{\lambda_i} x_{i+1}^{\lambda_{i+1}} \cdots x_k^{\lambda_k} \\
& + \sum_{i=1}^m (-1)^i x_0^{\lambda_1} \cdots x_{i-2}^{\lambda_{i-1}} x_i^{\lambda_i} \cdots x_k^{\lambda_k} \\
& + (-1)^{m+1} x_0^{\lambda_1} \cdots x_{m-1}^{\lambda_m} x_{m+1}^{\lambda_{m+1}} \cdots x_k^{\lambda_k}.
\end{align*}
The $i$th term of the first sum cancels with the $(i+1)$th term of the second, so the expression telescopes and all unmixed terms vanish.
\end{proof} 

If we can find polynomials for which $f(a + b) = f(a) + f(b)$, then we can immediately apply \ref{OnlyMixedTerms} to demonstrate that these polynomials do in fact lie in $\ker \coboundarymap_m$.  If we restrict our attention to working in a ring of characteristic $p \ne 0$, then there are a few obvious examples of such polynomials.

\begin{cor}\label{InitialPOPAreCocycles} 
Asymmetric monomials associated to multi-indices of the form $(p^{a_1}, \ldots, p^{a_m}, b_{m+1}, \ldots, b_k)$ are $m$-cocycles in a coefficient ring of characteristic $p$.
\end{cor}
\begin{proof}
For any $a, b$ in a ring of characteristic $p$, recall that $(a+b)^p = a^p + b^p$.  This then follows immediately from \ref{OnlyMixedTerms}.
\end{proof} 

\begin{cor}\label{POPAreCocycles} 
The symmetrized polynomial $\tau(\lambda)$ for $\lambda$ a power-of-$p$ multi-index is a cocycle under $\coboundarymap_m: \Z_p[{\bf x}] \to \Z_p[x_0, {\bf x}]$.
\end{cor}
\begin{proof}
Each element of this sum is a cocycle by \ref{InitialPOPAreCocycles}, and $\coboundarymap_m$ is a linear map.
\end{proof} 

We thus have a few critical examples of symmetric cocycles in the modular case of $\Z_p[{\bf x}]$.  We will also require some rational cocycles, which have been classified previously in \cite{AHS01}.  We reconstruct what we will need here.

\begin{defn}\label{zetaDefn} 
$\zeta_k^n \in \Z[{\bf x}]$ denotes the polynomial $(i \coboundarymap_1)^{k-1} x^n$ when divided by the gcd of the resultant monomial coefficients, where $i: \Z[x_0, \ldots] \to \Z[x_1, \ldots]$ acts by $i: x_i \mapsto x_{i+1}$ and exponentiation denotes repeated application.  Because $\coboundarymap$ is a differential, this is a $2$-cocycle.\end{defn} 

The authors of \cite{AHS01} go on to demonstrate that $\ker \coboundarymap_2$ (for $\coboundarymap_2: \Z[{\bf x}] \to \Z[x_0, {\bf x}]$) is in fact generated by these $\zeta_k^n$; the reader interested in a classification of the integral cocycles can find a proof at the beginning of \cite{AHS01}'s appendix A.  $\zeta_k^n$ is an interesting polynomial on its own; when we expand the $\coboundarymap_1$ applictions in \ref{zetaDefn}, we find the expression takes the form: \[\zeta_k^n =  d^{-1} \sum_{\substack{X \subseteq \{x_1, \ldots, x_k\} \\ X \ne \emptyset}} \left( (-1)^{|X|}\cdot\left(\sum_{x \in X} x\right)^n\right),\]  for some $d \in \Z$.  When the sums are expanded, we find \[\zeta_k^n = \left(\gcd_{0 \not\in \lambda} {n \choose \lambda}\right)^{-1}\sum_{\substack{0 \not \in \lambda \\ \lambda_i \ge \lambda_{i+1}}}{n \choose \lambda}\tau(\lambda).\]  Using this second expansion, we see immediately from the formal definition of $p$-carry-count that monomial summands of $\zeta_k^n$ of the form $c_\lambda {\bf x}^\lambda$ belonging to non-carry-minimal $\lambda$ will vanish under $\pi_p$ while carry-minimal $c_\lambda$ will remain non-zero.


\subsection{Gathering}\label{Gathering}

Beginning with the symmetric polynomials guaranteed to us to be cocycles by \ref{POPAreCocycles}, we investigate how to modify and extend these to form new cocycles.  Looking to previous classifications for clues, \cite{AHS01} employs the fact that their integral cocycles $\zeta_k^n$ form modular cocycles when their coefficient ring $\Z$ is projected down to $\Z_p$.  It is obvious that the following diagram commutes for a general coefficient ring $A$ and ideal $I \subseteq A$:
\begin{diagram}
A[{\bf x}]     & \rTo^{\coboundarymap_m} & A[{\bf x}, x_0]     \\
\dTo >{\pi_I}  &                         & \dTo >{\pi_I}       \\
(A/I)[{\bf x}] & \rTo^{\coboundarymap_m} & (A/I)[{\bf x}, x_0]
\end{diagram}

Noting that the $m$-coboundary map when applied to $A[{\bf x}]$, $\length{{\bf x}} = k$ leaves the remaining $k - m$ variables undisturbed (i.e., for $i > m$ the evaluation maps in $\coboundarymap_m$ act on $x_i$ by sending it to $x_i$), we can decompose the polynomial ring $A[x_1, \ldots, x_k]$ into the ring extension $(A[x_{m+1}, \ldots, x_k])[x_1, \ldots, x_m]$, bringing the variables undisturbed by the cocycle condition into the coefficient ring.  This effectively rewrites a $k$-variable $m$-cocycle $f$ as \[f = \sum_{\length{I} = k - m} (x_{m+1}, \ldots, x_k)^I \cdot f_I(x_1, \ldots, x_m),\] where each $f_I$ is an $m$-variable $m$-cocycle.

In this new coefficient ring we have a wide range of nontrivial ideals to select; picking ideals of the form $I = \langle x_i - x_j \rangle$ with $m < i < j \le k$ will take $k$-variable symmetric cocycles to $(k-1)$-variable asymmetric cocycles, in effect giving us approximate information about the lower dimensional cases.  $\pi_I$ is the operation that we call ``gathering.''  Denoting $k$-dimensional $m$-cocycles over $A$ as $Z^k(A)$ and their symmetric subset as $Z^k_*(A)$, the following (noncommutative) diagram paints a portrait of what we have so far:

\begin{diagram}
Z^{m+1}_*(\Z)   &           & \rTo_{\coboundarymap_1} &               & Z^{m+2}_*(\Z)   &           & \rTo_{\coboundarymap_1} &               & Z^{m+3}_*(\Z)   & \rTo     & \cdots \\
                &           &                         &               &                 &           &                         &               &                 &          &        \\
\dTo_{\pi_p}    &           & Z^{m+1}(\Z_p)           &               & \dTo_{\pi_p}    &           & Z^{m+2}(\Z_p)           &               & \dTo_{\pi_p}    &          & \cdots \\
                & \ldDashto &                         & \luTo^{\pi_I} &                 & \ldDashto &                         & \luTo^{\pi_I} &                 &          &        \\
Z^{m+1}_*(\Z_p) &           & \lDashto                &               & Z^{m+2}_*(\Z_p) &           & \lDashto                &               & Z^{m+3}_*(\Z_p) & \lDashto & \cdots
\end{diagram}

In this section we seek to construct the dashed maps, and to do so understanding the exact nature of $\pi_I$'s action will help.  Selecting $\tau(9, 1, 1, 1)$ as an example in the case $k = 4$ and $p = 3$, its gathering is \[\pi_I \tau(9, 1, 1, 1) = x_1^9 x_2 x_3^2 + x_1 x_2^9 x_3^2 + 2 x_1 x_2 x_3^{10}\] for $I = \langle x_3 - x_4 \rangle \subseteq (\Z_3[x_3, x_4])[x_1, x_2]$.  Obviously this polynomial is no longer symmetric, and so we seek to find an appropriate symmetrization (i.e., an action for the dashed map).  It is almost immediately obvious that the na\"ive symmetrization $\sum_{\sigma \in S_{k-1}} (\pi_I f)(\sigma {\bf x})$ will not be sufficient in general (see appendix \ref{Char5Table} for a plethora of complicated low-dimensional examples), and we must be more creative.

The other information we have at this point in dimension $k-1$ is that $\pi_p \zeta_{k-1}^n$ is a symmetric cocycle.  As we will shortly prove, gathering cocycles sufficiently ``near'' a power-of-$p$ cocycle $\tau \lambda$ yields a sum of monomials $\sum_i {\bf x}^{S_i}$ for an indexed set of multi-indices $S$ where each $S_i$ is carry-minimal.  (For example, this holds true in the above case; both $(9, 2, 1)$ and $(10, 1, 1)$ have minimal $3$-carry.)  This means that these monomials appear as summands of $\zeta_k^n$ (for instance, $\zeta_3^{12} = \tau(9, 2, 1) - \tau(10, 1, 1) + \tau(6, 3, 3)$), and so we may be able to use the information contained there to recover symmetric cocycles corresponding to the gathered cocycles.  We begin by formalizing the notion of ``nearness.''

\begin{defn}\label{SplittingDistance} 
Define the function $\phi_p$ from multi-indices to $\N_0$ by \[\phi_p: \lambda \mapsto \sum_{i = 1}^{\length{\lambda}} \sigma_p(\lambda_i) - \length{\lambda} = \ell\left(\bigcup_{i = 1}^{\length{\lambda}} \rho_p(\lambda_i) \right) - \length{\lambda},\] which corresponds to the fewest number of gathering operations required to reach $\lambda$ from a power-of-$p$ multi-index.  $\phi_p(\lambda)$ is called the \emph{splitting distance} of $\lambda$.
\end{defn} 

For example, working in $p = 3$, we have
\begin{align*}
\phi_3(9, 1, 1, 1) & = 0, \\
\phi_3(9, 2, 1) & = 1, \\
\phi_3(10, 2) & = 2, \\
\phi_3(9, 3) & = 0.
\end{align*}

\begin{thm}\label{LargeSplitCausesCarry} 
If $\lambda$ is a multi-index such that $\phi_p(\lambda) > p - 2$ and $\alpha_p(\lambda) > 0$, then $\lambda$ is not carry minimal.
\end{thm}
\begin{proof}
We define a gathering operator, $G_{ij}$, of a multi-index $\lambda$, $1\le i < j\le\length{\lambda}$, to be \[G_{ij}\lambda=(\lambda_i + \lambda_j) \cup (\lambda \setminus (\lambda_i, \lambda_j)).\]  Note that $\alpha_p(G_{ij}\lambda) \le \alpha_p(\lambda)$ for all gathering operators $G_{ij}$.

Let ${\hat \lambda} = \bigcup_i \rho_p(\lambda_i)$.  Because $\alpha_p({\hat \lambda}) = \alpha_p(\lambda)> 0$ and ${\hat \lambda}$ is a power-of-$p$ multi-index, ${\hat \lambda}$ contains $p$ many copies of some $p^k$.  These $p^k$ can be gathered in $p-1$ steps to form a multi-index $\mu$ that satisfies $\alpha_p(\mu) < \alpha_p(\lambda)$.  We can then apply any gathering operations we like to $\mu$ to achieve a multi-index $\mu'$ with $\length{\mu'} = \length{\lambda}$, and we are still guaranteed that $\alpha_p(\mu') \le \alpha_p(\mu) < \alpha_p(\lambda)$, which means that $\lambda$ cannot be $p$-carry minimal.
\end{proof} 


This technique of forming an alternative gathering guides the structure of many of the remaining proofs.  Using it, we can immediately gain various facts about the carry-minimality of power-of-$p$ multi-indices and their gatherings.

\begin{cor}\label{ExhaustivenessOfPOP} 
For $n$, $k$, $p$ such that there exists a power-of-$p$ partition $\mu \vdash n$ with $\length{\mu} = k$, then any $\lambda$ of the same weight and length that is not power-of-$p$ will not be carry minimal.
\end{cor}
\begin{proof}
First, note that all partitions can be reached via ``gathering'' (i.e., applying raising operators to) the trivial partition \[\stackrel{\hbox{$n$ times}}{\overbrace{(1, \ldots, 1)}} = \bigcup_{i=1}^n (1) = (1^n)\] of weight and length $n$, and that in particular we require exactly $p-1$ gathering operations to collect $p$ copies of any $p^k$ to form an instance of $p^{k+1}$.  Thus all power-of-$p$ partitions occur at regular intervals of length $(p-1)$.

Now, let $\lambda$ be a non-power-of-$p$ multi-index, and $\mu$ be power-of-$p$ with equal weight and length.  Then ${\hat \lambda} = \bigcup_i \rho_p(\lambda_i)$ is power-of-$p$ with $\length{\hat \lambda} > \length{\mu}$, and so $\phi_p(\lambda) \ge p-1$.  We can then apply \ref{LargeSplitCausesCarry}.
\end{proof} 

\begin{cor}\label{CarryMinGathersAreCarryMin} 
Iteratively gathering a power-of-$p$ monomial $\tau(\lambda)$ results in a sum of carry-minimal exponent monomials when done fewer than $p-1$ times.
\end{cor}
\begin{proof}
Clearly gathering power-of-$p$ partitions fewer than $p-1$ times will result in partitions of the same carry-count; the question is whether this count is still minimal for the length.  There are two cases: one where $\alpha_p(\lambda) = 0$ and one where $\alpha_p(\lambda) > 0$.  For the first, all gatherings of $\lambda$ will also have carry-count $0$, and so they are trivially carry minimal.  Now, let $\alpha_p(\lambda) > 0$, let $\mu$ be the result of fewer than $p-1$ gathering operations applied to $\lambda$, and assume $\nu$ is a carry-minimal with the same weight and length as $\mu$ such that $\alpha_p(\nu) < \alpha_p(\mu)$.  $\nu$ cannot arise as a gathering of anything which is power-of-$p$ and the same weight and length as $\lambda$ (since then $\alpha_p(\nu) = \alpha_p(\lambda) = \alpha_p(\mu)$), and so $\length{\bigcup_i \rho_p(\nu_i)} > \length{\lambda}$, which implies $\phi_p(\nu) > p-2$.  We can then apply \ref{LargeSplitCausesCarry}, so $\nu$ is not carry minimal, a contradiction.
\end{proof} 

\begin{cor}\label{SimilarityOfPOP} 
Let $\lambda, \mu$ be two power-of-$p$ partitions of equal length and weight.  Then $\alpha_p(\lambda)=\alpha_p(\mu)$.
\end{cor}
\begin{proof}
By \ref{CarryMinGathersAreCarryMin}, all power-of-$p$ partitions are carry minimal, hence if $\lambda$, $\mu$ are power-of-$p$ of equal length and weight, $\alpha_p(\lambda)\le\alpha_p(\mu)$ and vice-versa.
\end{proof} 

Now we return to symmetrizing $m$-fold gatherings of $\tau \lambda$, for $\lambda$ power-of-$p$ and $m < p - 1$.  We begin with something slightly weaker; we find a symmetric cocycle in which all our symmetrized monomials appear as summands (that is to say that the cocycle will consist of our monomials plus an extension), and later on we'll demonstrate that we may simply drop the extension.  Since we know that the monomials visible after performing such an operation have corresponding partitions which are carry-minimal, we can simply steal directly from $\pi_p \zeta_k^n$ for appropriate $k$ and $n$.

\begin{lem}\label{GathersAdmitExtensions} 
Let $\lambda$ be a power-of-$p$ partition and either $m < p - 1$ or $\lambda = \rho_p(\weight{\lambda})$.  Then the monomials in the image of gathering $\tau \lambda$ $m$-many times can be resymmetrized, assigned non-zero coefficients, and extended by other symmetrized monomials such that the resultant linear combination is a symmetric $2$-cocycle.
\end{lem}
\begin{proof}
By \ref{CarryMinGathersAreCarryMin}, we have that the monomials resultant from gathering the initial monomial are carry minimal, and so $\pi_p \left( \zeta_{\length{\lambda}-m}^{\weight{\lambda}} \right)$ is such an extension.
\end{proof} 

This is not quite enough to meet our original goal, since we may be forced to add other symmetrized monomials beyond what we expect from gathering.  Consider, again, gatherings of $(9, 1, 1, 1)$.  We find $\pi_3 \zeta_3^{12} = \tau(9, 2, 1) - \tau(10, 1, 1) + \tau(6, 3, 3)$, and so $\tau(6, 3, 3)$ is an unwanted extension.  $\tau(6, 3, 3)$, however, occurs as the $1$-fold gathering of $\tau(3, 3, 3, 3)$, which is a distinct power-of-$p$ monomial of the same degree and dimension as $\tau(9, 1, 1, 1)$.  We turn our attention toward using this observation to separate out parts of the projected integral cocycle, each of which on its own is a modular cocycle.

Fix natural $n$ and $k$ where a power-of-$p$ multi-index $\lambda$ of weight $n$ and length $k$ exists.  Let \[T^0(n,k) = \left\{ \{ \lambda\} \mid \hbox{$\lambda \vdash n$ is power-of-$p$}, \length{\lambda} = k \right\}.\]  Let $\theta$ be a map from sets of partitions to sets of their single-step gatherings, and let $T^m(n,k)$ be inductively defined as $\{\theta M \mid M \in T^{m-1}(n,k)\}$.  When the context is clear, we drop $(n, k)$ and write only $T^m$.  We then seek to decompose $\Z_p[x_0, {\bf x}]$ into a sum of submodules $\bigoplus_i A_i$ such that there is a one-to-one correspondence between the $A_i$ and the elements of our particular $T^m$.  The decomposition must satisfy that for each $M_i \in T^m$ with associated component $A_i$, we have $\coboundarymap_2 \tau\mu \in A_i$ for each $\mu \in M_i$, which guarantees linear independence of the gatherings of the various power-of-$p$ symmetrized monomials.

For $(9, 1, 1, 1)$, we have $n = 12$, $k = 4$, and $p = 3$.  Here, we compute \[T^0 = \left\{\{(9, 1, 1, 1)\}, \{3, 3, 3, 3\}\right\},\] \[T^1 = \theta T^0 = \left\{\{(9, 2, 1), (10, 1, 1)\}, \{(6, 3, 3)\}\right\}.\]  Ordering $T^1$ as written above, one choice of $A_1$ is $\mspan \{\tau(9, 1, 1, 1)\}$.  A matching choice for $A_2$ is $\mspan (B_3^{12} \setminus \{\tau(9, 1, 1, 1)\})$, which has $\mspan \{\tau(3, 3, 3, 3)\}$ as a subspace, and the relation to preimages by $\theta$ is not coincidental.  To use this observation, we must first demonstrate that distinct sets of gatherings $M_i, M_j \in T^m$ are disjoint for $m < p - 1$, and then we may use our knowledge of basic polynomial arithmetic in $\Z_p[{\bf x}]$ to show that such extensions are not necessary.

\begin{lem}\label{GatheringIntersection} 
Given two distinct power-of-$p$ multi-indices $\lambda$ and $\mu$ both of weight $n$ and length $k$, gathering must be applied at least $p-1$ times before their gatherings have nonempty intersection.
\end{lem}
\begin{proof}
Let $\lambda$ be a carry-minimal multi-index with gathered from ${\hat \lambda} = \bigcup_i \rho_p(\lambda_i)$, and let $\mu$ be some other gathering parent of $\lambda$ of the same weight and length as ${\hat \lambda}$.  We know by \ref{ExhaustivenessOfPOP} that $\alpha_p(\mu) = \alpha_p({\hat \lambda}) = \alpha_p(\lambda)$, and so gathering cannot combine $p$ copies of $p^k$ into any single entry of $\lambda$.  Hence $\mu$ contains as many copies of $p^k$ for any particular $k$ as ${\hat \lambda}$.  Since every $\mu$ and $\hat{\lambda}$ are contain only powers of $p$, $\mu = {\hat \lambda}$.
\end{proof} 

In $\coboundarymap_2 \tau(10, 1, 1)$, we see summands such as $(x_0 + x_1)^{10} x_2 x_3$.  These, in turn, have expansions of the form $x_0^{10} x_2 x_3 + x_0^9 x_1 x_2 x_3 + x_0 x_1^9 x_2 x_3 + x_1^{10} x_2 x_3$ in $\Z_3[{\bf x}]$.  The unmixed terms cancel, hence we need only consider the middle two summands, which take the remarkable form of previous gatherings of $\tau(10, 1, 1)$'s power-of-$p$ parent.

\begin{thm}\label{GathersNeedNoExtensions} 
Under the same conditions as \ref{GathersAdmitExtensions}, monomials in the image of gathering of $\tau \lambda$ $m$-many times can be resymmetrized and assigned non-zero coefficients such that the result is a $2$-cocycle.
\end{thm}
\begin{proof}
This monomial can be gathered into monomials $m'_1, m'_2, \ldots, m'_l$ with symmetrizations $m_1, \ldots, m_l$.  By \ref{GathersAdmitExtensions} these can be assigned non-zero coeffcients $c_1, \ldots, c_l$ and extended by some $f$ such that $c_1 m_1 + \cdots + c_l m_l + f$ forms a $(k-1)$-dimension cocycle and $f$ contains no monomials that appear in $m_i$.  We argue that $f$ can always be chosen to be 0.

Recall that if $c$'s base-$p$ representation is $\rho_p(c) = (c_1, \ldots, c_l)$, then for all $a, b \in \Z_p$ we have \[(a+b)^c = \prod_{i \in \{1, \ldots, l\}} (a+b)^{c_i} = \prod_{i \in \{1, \ldots, l\}} (a^{c_i} + b^{c_i}) = \sum_{S \subseteq \{1, \ldots, l\}} a^{\sum_{i \in S} c_i} b^{\sum_{i \not\in S} c_i}.\]  Using this, the cocycle condition applied to a carry-minimal $\lambda$ with parent ${\hat \lambda} = \bigcup_i \rho_p(\lambda_i)$ will then result in a sum of monomials of the form $\sum_\mu c_\mu {\bf x}^\mu$, $c_\mu \ne 0$, whose exponents $\mu$ are either a reordering of the multi-index $\lambda\cup\{ 0\}$ or a gathering of ${\hat \lambda}$ whose length is $\length{\lambda}+1$.  The $f$ given by the residual terms of $\pi_p \zeta_k^n$ is composed of $(\phi(\lambda)-1)$-fold gatherings of other power-of-$p$ symmetrized monomials of the same weight and length as ${\hat \lambda}$.  By \ref{GatheringIntersection} we then have that the images generated by each power-of-$p$ cocycle are linearly independent under $\coboundarymap_2$, and so $\coboundarymap_2(c_1 m_1 + \cdots + c_l m_l + f) = 0$ implies that $\coboundarymap_2(c_1 m_1 + \cdots c_l m_l) = 0$.
\end{proof} 

\subsection{Integral Projection}\label{IntegralProjection}

In \textsection \ref{Gathering} we demonstrated the existence of a wide variety of modular cocycles, using power-of-$p$ multi-indices and the existence of a particular rational cocycle as input.  We now show that this exhausts all possible $2$-cocycles.  Such a statement has two parts: there are no cocycles that do not occur via this process, and cocycles that do occur as part of this process cannot be written as the sum of two ``smaller'' cocycles for an appropriate interpretation of the word ``smaller.''  These will actually be proven nearly simultaneously, but we must first frame the question appropriately, beginning by precisely communicating a notion of ``smallness,'' which we more suggestively name ``indecomposable.''

\begin{defn}\label{DecomposableDefn} 
A $k$-variable $m$-cocycle $f$ of degree $n$ is called \emph{decomposable} if there exists a set partition $\{B_1, B_2\}$ of $B_k^n$ with some $f_1 \in \mspan B_1, f_2 \in \mspan B_2$ such that $f_1 + f_2 = f$ and $f_1, f_2 \in \ker \coboundarymap_m$.  $f$ is called \emph{indecomposable} otherwise.
\end{defn} 

\begin{lem}\label{IndecomposableBasis} 
The set of indecomposable cocycles is a basis for the kernel of $\coboundarymap_m$ taken over any field.
\end{lem}
\begin{proof}
First, homogenous $f \in \ker \coboundarymap_m$ can be written as the sum of indecomposable cocycles.  Note first that for indecomposable $f$ and for symmetrized monomials this is trivially true.  Assume instead that $f$ is decomposable, let $B_f \subseteq B_k^n$ be such that $f \in \mspan{B_f}$, and assume that the lemma holds for all $B_{f'} \subset B_f$.  Let $B_1, B_2, f_1, f_2$ be as in \ref{DecomposableDefn}.  The inductive hypothesis gives us the existence of $f_1 = \sum_i f_{1,i}$ and $f_2 = \sum_j f_{2,j}$, where $f_{1,i}, f_{2,j}$ are indecomposable cocycles.  The desired decomposition is then $f = \sum_i f_{1,i} + \sum_j f_{2,j}$, and inducting over the size of $B_f$ shows that the set of indecomposables spans $\ker \coboundarymap_m$.

In addition, the set of indecomposables is linearly independent.  Assume instead that two indecomposable cocycles $f_1, f_2$ share a particular monomial $f = \sum_{\sigma \in S_k} {\bf x}^{\sigma \lambda}$ with coefficient $b_1$ in $f_1$ and $b_2$ in $f_2$.  Then $b_2 f_1 - b_1 f_2$ is a cocycle with a zero coefficient on $f$, and $b_2^{-1}(b_2 f_1 - b_1 f_2) + b_1 f$ is a decomposition of $f_1$.
\end{proof} 

When we apply $\coboundarymap_m$ to a particular symmetrized monomial $\tau \lambda$ with $\lambda$ not power-of-$p$, we see a sum of image monomials.  In order for $\tau \lambda$ to participate in a cocycle, we must include other symmetrized monomials with which we might cancel the image of $\tau \lambda$ to achieve zero.  Given an image monomial ${\bf x}^{(\lambda_0, \ldots, \lambda_k)}$, in the general case of $\coboundarymap_m$, the possible parent monomials must be of the form (we will prove this in a moment) ${\bf x}^{(\lambda_0, \ldots, \lambda_i + \lambda_{i+1}, \ldots, \lambda_k)}$ for some $0 \le i < m \le k$.  When $m = 2$, we do not have a choice; one of the two parents belongs to the symmetrized monomial we are trying to cancel, and so the other parent is our only choice and we are forced to include it.  We can iterate this process on this new summand, and the collection of such multi-indices we call the annihilating set of $\lambda$, which is formally defined as follows:

\begin{defn}\label{AnnihilatorDefn} 
We define the map $\theta$ from partitions to sets of partitions by the following rule:  if $\lambda$ is a partition of length $k$, then for any permutation $\sigma \in S_k$ and corresponding reordering $\sigma \lambda = {\tilde \lambda} = ({\tilde \lambda_1}, {\tilde \lambda_2}, \ldots, {\tilde \lambda_k})$, we have $\left(\tilde \lambda '_1, \tilde \lambda ''_1 + {\tilde \lambda_2}, {\tilde \lambda_3}, \ldots, {\tilde \lambda_k}\right) \in \theta(\lambda)$ exactly when $\rho_p({\tilde \lambda'_1}) \cup \rho_p({\tilde \lambda ''_1}) = \rho_p({\tilde \lambda_1})$.  Define $\Theta(S) = \bigcup_{s \in S} \theta(s)$.  The sequence $\Theta^n(\{\lambda\})$ is nondecreasing and bounded, and is thus eventually constant with value denoted $\ann \lambda$, called the \emph{annihilating set of $\lambda$}.
\end{defn} 

This definition takes into account our observation from \ref{GathersNeedNoExtensions} concerning $(a + b)^c$ for $a, b \in \Z_p$; we need not consider splittings ${\tilde \lambda_1'} + {\tilde \lambda_1''} = {\tilde \lambda_1}$ that do not occur as $\rho_p(\lambda_1') \cup \rho_p(\lambda_1'') = \rho_p ({\tilde \lambda_1})$.  We compute some sample annihilator sets of degree $12$, dimension $3$, and characteristic $3$ below:
\begin{align*}
\ann (9, 2, 1) & = \{(9, 2, 1), (10, 1, 1)\}, \\
\ann (6, 3, 3) & = \{(6, 3, 3)\}, \\
\ann (4, 4, 4) & = \{(4, 4, 4), (5, 4, 3), (6, 4, 2), (6, 5, 1), (7, 3, 2), (7, 4, 1), (8, 3, 1)\} \cup \ann (9, 2, 1) \cup \ann (6, 3, 3), \\
\ann (5, 5, 2) & = \{(5, 5, 2), (8, 2, 2)\} \cup \ann (4, 4, 4).
\end{align*}
Now, we demonstrate that these sets actually capture what we want:

\begin{lem}\label{AnnihilatorIsAptlyNamed} 
Any linear combination of symmetrized monomials lying in $\ker \coboundarymap_2$ containing $\tau(\lambda)$ for some partition $\lambda$ will also contain each of $\tau(\lambda ')$ for $\lambda ' \in \ann \lambda$.
\end{lem}
\begin{proof}
$\coboundarymap_2(\tau \lambda)$ will contain ${\bf x}^\mu$, $\mu = \left(\lambda '_1, \lambda ''_1, \lambda_2, \lambda_3, \ldots, \lambda_k\right)$ for every $\rho_p(\lambda'_1) \cup \rho_p(\lambda''_1) = \rho_p(\lambda_1)$ by the reduced binomial expansion noted in \ref{GathersNeedNoExtensions}.  The preimage of ${\bf x}^\mu$ by $\coboundarymap_2$ contains at most $\lambda$ and the partition \[\left(\lambda'_1, \lambda''_1 + \lambda_2, \lambda_3, \ldots, \lambda_k\right) = \nu,\] the latter when $\alpha_p \left( \lambda_1'', \lambda_2 \right) = 0$.  This is because the third term of $\coboundarymap_2 (\tau \lambda')$ will be of the form \[x_0^{\lambda '_1}(x_1 + x_2)^{\lambda''_1 + \lambda_2}x_3^{\lambda_3}\cdots x_k^{\lambda_k}.\]  Therefore if ${\bf x}^\mu$ is to vanish then $\tau \nu$ must appear in linear combination with $\tau \lambda$.  This is exactly the definition of $\nu \in \ann \lambda$.
\end{proof} 

Problems arise when $\rho_p(\lambda''_1) \cup \rho_p(\lambda_2) \ne \rho_p(\lambda''_1 + \lambda_2)$, since then $\nu$ cannot contribute the requisite cancelling monomial to the image.  This happens when $\alpha_p(\lambda ''_1, \lambda_2) > 0$, or equivalently when $\alpha_p(\nu) < \alpha_p(\lambda)$.  This happens strikingly often; suppose we begin with the partition $(4, 4, 4)$ and set $p = 3$.  Then we can split $\rho_3(4) = (3, 1)$ as $(1) \cup (3)$ and form an element $(1, 3 + 4, 4) = (1, 7, 4) \in \theta(4, 4, 4) \subseteq \Theta^1(\{(4, 4, 4)\})$.  $(1, 7, 4)$ can be reordered as $(4, 7, 1)$ and then $\rho_3(4)$ can again be split as $(1) \cup (3)$, giving an element $(1, 3 + 7, 1) = (1, 10, 1) \in \theta(4, 7, 1) \subseteq \Theta^2(\{(4, 4, 4)\}) \subseteq \ann (4, 4, 4)$.  Since $\alpha_3(10, 1, 1) = 1$ and $\alpha_3(4, 4, 4) = 2$, we have constructed our desired $\nu$ with $\alpha_p(\nu) < \alpha_p(\lambda)$ and $\nu \in \ann \lambda$.  This same game can actually be played with every non-carry-minimal partition, which gives us the first half of our exhaustiveness argument:


\begin{thm}\label{AnnihilatorLowersCarry} 
If $\lambda$ is a non-carry-minimal partition in base $p$ and has carry-count $\alpha_p(\lambda)$, there exists an $\lambda ' \in \ann \lambda$ with $\alpha_p(\lambda ') < \alpha_p(\lambda)$.
\end{thm}
\begin{proof}
The nearest (in terms of splitting distance) power-of-$p$ partition for which $\lambda$ may be gathered is given by $\bigcup_i \rho_p(\lambda_i) = {\hat \lambda}$.  Because $\lambda$ is not carry minimal, there exists--as in the proof of \ref{LargeSplitCausesCarry}--a partition $\mu$ which is power-of-$p$ with $\length{\lambda} \le \length{\mu} < \length{\hat \lambda}$ and $\weight{\lambda} = \weight{\mu}$ that is a $(p-1)$-fold gathering of ${\hat \lambda}$.  Let $p^k$ be a power of $p$ disturbed in the gathering process to form $\mu$ from ${\hat \lambda}$.  Noting that there must be at least $p$ copies of $p^k$ present in ${\hat \lambda}$, we can iteratively separate out the copies of $p^k$ in our original partition, $\lambda$.

Borrowing the notation of the construction from \ref{AnnihilatorIsAptlyNamed}, we begin by permuting $\lambda$ such that $p^k \in \rho_p(\lambda_{\sigma 1})$, then taking $\lambda '_1$ to be all the copies of $p^k$ in $\rho_p(\lambda_{\sigma 1})$ and $\lambda ''_1$ to be everything else (i.e., $\lambda '_1 + \lambda ''_1 = \lambda_{\sigma 1}$), then turning our attention to $(\lambda '_1, \lambda ''_1 + \lambda_2, \ldots, \lambda_k) \in \ann \lambda$.  (Of course, if $\rho_p(\lambda_{\sigma 1}) = \bigcup_{i=1}^j (p^k)$ for some $j$, we must leave one $p^k$ in $\lambda ''_1$ because $\lambda ''_1$ cannot equal zero.)  After permuting by the cycle $\sigma = (1\, 2)$, we then call this freshly constructed partition $\lambda$.

At each step we select some $\sigma \in S_{\length{\lambda}}$ such that $2$ is undisturbed and $p^k \in \rho_p((\sigma \lambda)_1)$.  We then reverse the above construction, splitting the sum of all copies of $p^k$ present in $(\sigma \lambda)_1$ into $\lambda ''_1$ and the remainder into $\lambda '_1$, then collecting $\lambda ''_1$ with the $p^k$ accumulating in $\lambda_2$, each time generating a new multi-index $(\lambda '_1, \lambda ''_1 + \lambda_2, \ldots, \lambda_{\sigma k})$ that lies in $\ann \lambda$ (with the same caveat about $\lambda_1 = \bigcup_{i=1}^j (p^k)$ for some $j$).  After at most $p$ many steps, $\lambda_2 \ge p^{k+1}$, which constructs a partition in $\ann \lambda$ with carry-count reduced by $1$.
\end{proof} 

The second half of the argument lies in noting the following invariant of $\coboundarymap_m$:

\begin{lem}\label{DifferentialPreservesCarry} 
Let $\lambda$ be a partition, and select $\mu$ such that the coefficient $c_\mu$ is nonzero in \[\coboundarymap_m(\tau \lambda) = \sum_{\mu} c_\mu \cdot (\tau \mu).\]  Then $\alpha_p(\lambda) = \alpha_p(\mu)$.
\end{lem}
\begin{proof}
Let $\lambda '_1, \lambda ''_1$ be such that $\mu = (\lambda \setminus (\lambda_1)) \cup (\lambda'_1, \lambda''_1)$.  Then, \[ \alpha_p(\mu) = \alpha_p(\lambda '_1, \lambda ''_1, \lambda_2, \ldots, \lambda_k) = \alpha_p(\lambda '_1, \lambda ''_1) + \alpha_p(\lambda_1, \ldots, \lambda_k) = 0 + \alpha_p(\lambda),\] where the last equality stems from noticing that $\rho_p(\lambda'_1) \cup \rho_p(\lambda''_1) = \rho_p(\lambda_1)$, as remarked upon in \ref{GathersNeedNoExtensions}.
\end{proof} 

\begin{thm}\label{NoNonCarryMinimals} 
If $\lambda$ is not a $p$-carry-minimal partition, then $\tau \lambda$ cannot participate in a cocycle.
\end{thm}
\begin{proof}
By \ref{AnnihilatorLowersCarry}, we have that each non-carry-minimal partition's annihilating set contains another of strictly lower carry-count.  If we follow the construction of $\ann \lambda$ using definition \ref{AnnihilatorDefn}, there must exist partitions $\lambda', \lambda'' \in \ann \lambda$ such that $\lambda''$ is required to cancel an image monomial of $\lambda'$, and $\alpha_p(\lambda'') < \alpha_p(\lambda')$.   Since by \ref{DifferentialPreservesCarry} all of the monomials in $\coboundarymap_2 \tau \lambda''$ have carry-counts distinct from those of $\coboundarymap_2 \tau \lambda'$, they cannot cancel each other, and in turn $\tau \lambda$ cannot participate in a cocycle.
\end{proof} 

\begin{cor}\label{RationalProjection} 
Let $\{\beta_i\}_i$ be the indecomposable basis associated to the subspace of cocycles of dimension $k$, degree $n$, and characteristic $p$.  Then $\pi_p(\zeta_k^n) = \sum_i a_i \beta_i$, where the $a_i$ are all non-zero.
\end{cor}
\begin{proof}
Immediate from the alternative expansions of $\zeta_k^n$ noted after \ref{zetaDefn} and its decomposition into a sum of indecomposables.
\end{proof} 

In addition to the exhaustiveness above, we can also use these annihilating sets to achieve indecomposability of resymmetrized gatherings of \ref{GathersNeedNoExtensions}.

\begin{thm}\label{AnnihilatorSetIsTmnk} 
Let $T$ be a set in $T^m(n, k)$ (as defined in \textsection\ref{Gathering}), where $m$, $n$, and $k$ satisfy the conditions of \ref{GathersNeedNoExtensions}.  Then for every $\lambda \in T$, $\ann \lambda = T$.
\end{thm}
\begin{proof}
For a fixed weight $n$ and length $k$, we define a function $d$ on unordered partitions of this type that takes a pair $(\lambda, \mu)$ to the number of slots in which $\lambda$ and $\mu$ differ.  Then, for distinct $\lambda, \mu \in T$, it suffices to show that there is a $\lambda' \in T \cap \ann \lambda$ such that $d(\lambda', \mu)$ is strictly less than $d(\lambda, \mu)$.  Under the conditions imposed in \ref{GathersNeedNoExtensions}, which mean to prevent any $p$ copies of a particular $p^k$ from being summed together during the gathering procedure used to form $T$, this is obvious.  Induction then yields that $\mu \in \ann \lambda$ and $\ann \lambda = T$.
\end{proof} 

\begin{cor}\label{GathersAreIndecomposable} 
The cocycles resulting from \ref{GathersNeedNoExtensions} are indecomposable.
\end{cor}
\begin{proof}
Immediate from \ref{AnnihilatorSetIsTmnk}, \ref{AnnihilatorIsAptlyNamed}, and \ref{GatheringIntersection}.
\end{proof} 

\subsection{Counting Additive $2$-Cocycles}\label{CountingAdditive2Cocycles}

In the end, we are studying additive cocycles with intent to eventually investigate the multiplicative cocycles, so that we may in turn apply these results to maps in algebraic topology.  The rank of these maps is related to the number of multiplicative cocycles present in a particular degree and dimension, which is in turn bounded from above by the number of additive cocycles present in the same degree and dimension (a statement made precise in \textsection\ref{TheFunctorspecHBU2k}).  The number-theoretic properties of the additive cocycles suggests a particular way to count them using generating functions, for which we give a construction below.

\begin{defn}\label{GeneratingFn} 
We define $C^p_{nk} \in \Z$ to be the coefficients of the generating function \[\prod_{i=0}^\infty (1 - t x^{p^i})^{-1} = \sum_{n, k} C^p_{nk} x^n t^k.\]
\end{defn} 

\begin{lem}\label{GFCountsPOP} 
$C^p_{nk}$ equals the number of ways to write $n$ as a sum of $k$ many powers of $p$.
\end{lem}
\begin{proof}
A factor of the product looks like \[\left(1 - t x^{p^i}\right)^{-1} = \sum_{j=0}^\infty t^j x^{j p^i}.\]  Expanding these products reveals that for particular $n$ and $k$, the summands contributing to $C^p_{nk}$ have the form $t^k x^n$ with $n = \sum_{m=1}^k a_m p^m$ for $a_m \in \N_0$.
\end{proof} 

These summations indeed correspond to cocycles, the proof of which is merely a summation of everything that's led to this point.

\begin{thm}\label{PowersOfPSpan} 
In a particular degree $n$ and number of variables $k$, if a power-of-$p$ multi-index exists, the restriction of $\coboundarymap_2$ to $k$-variable $n$-degree symmetric polynomials has kernel spanned by \[\left\{\tau(\lambda) \mid \hbox{$\lambda$ is a power-of-$p$ multi-index}, \length{\lambda} = k, \weight{\lambda} = n\right\}.\]
\end{thm}
\begin{proof}
Using \ref{ExhaustivenessOfPOP} and \ref{SimilarityOfPOP}, we have that exactly the power-of-$p$ multi-indices are $p$-carry minimal.  By our classification in \ref{RationalProjection} we have that they are exhaustive of all $2$-cocycles in the degree and dimension to which they belong, and because they are composed of single symmetrized monomials, they are trivially indecomposable.
\end{proof} 

\begin{cor}\label{NonzeroGFCountsCocycles} 
$C^p_{nk}$ count the number of $n$-degree $k$-dimensional cocycles when $C^p_{nk} \ne 0$.
\end{cor}
\begin{proof}
Immediate from \ref{GFCountsPOP} and \ref{PowersOfPSpan}.
\end{proof} 

When working in $\Z_2$, the generating function is especially nice, since every number $n$ has a power-of-$2$ representation of length $k$ for every $\sigma_2(n) \le k \le n$.

\begin{lem}\label{GFAlwaysNonzeroForP2} 
For $\sigma_2(n) \le k \le n$, $C_{nk}^2$ is always non-zero.
\end{lem}
\begin{proof}
As discussed in the proof of \ref{ExhaustivenessOfPOP}, power-of-$p$ multi-indices for a particular degree $n$ begin in dimension $n$ and occur for every dimension $n - c(p-1)$, $c \in \N_0$.  For $p = 2$, this means a power-of-$2$ multi-index occurs in every dimension in the range $\sigma_2(n), \ldots, n$.
\end{proof} 

However, the general case does not appear to be so well off, since there are lengths and weights for which no sum of powers of $3$ can be written.  For instance, $(3, 1, 1, 1)$ and $(3, 3)$ have lengths $4$ and $2$ respectively and both has weight $6$, but there exists no power-of-$3$ multi-index of weight $6$ and length $3$.  Pleasantly enough, because the gathering operation allows us to determine the presence of these intermediate cocycles knowing only what the kernel looks like in locations where power-of-$p$ multi-indices do exist, we can extend $C^p_{nk}$ to cover these middle cases as well.

\begin{thm}\label{ActualCocycleCount} 
Define $D^p_{nk}$ to be $C^p_{nk}$ when $C^p_{nk}$ is non-zero, to be $D^p_{n(k+1)}$ when $C^p_{nk}$ is zero and $k < n$, and $0$ otherwise.  Then $D^p_{nk}$ counts the number of cocycles in $\Z_p$ of degree $n$ and dimension $k$.
\end{thm}
\begin{proof}
This is just successive application of the results \ref{GatheringIntersection}, \ref{RationalProjection}, and \ref{GathersAreIndecomposable}.
\end{proof} 


\subsection{The Generalized Lazard Ring}\label{TheGeneralizedLazardRing}
Lazard demonstrated a ring isomorphism between the universal representing ring for two variable $2$-cocycles and a polynomial ring on countably many generators, a celebrated result in the theory of formal group laws.  Here we provide a similar result for the representing ring of $k$-variable $2$-cocycles.

To begin, since the representing ring selects cocycles over an arbitrary ring $A$, we must demonstrate that our knowledge about the $\Q$ and $\Z_p$ cases is sufficient to determine the rest of the story for arbitrary commutative rings.

\begin{thm}\label{ArbitraryRingDecomp} 
Let $A$ be an abelian group and let $f \in A \otimes \Z[\mathbf{x}]$ be a $k$-variable symmetric $2$-cocycle of homogenous degree $n$.  Then $f$ is of the form \[f = \sum_i a_i \otimes \beta_i(\mathbf{x}),\] where $a_i \in A$ and $\beta_i$ is the relevant indecomposable basis of \ref{IndecomposableBasis}.
\end{thm}
\begin{proof}
We begin by making a number of standard reductions.  First, since only finitely many terms will appear in $f$, it is sufficient to prove the theorem when $A$ is finitely generated.  Then, for two abelian groups $A \subseteq B$, if the theorem is true for $B$ then it is also true for $A$.  This implies in addition that if the theorem is true for arbitrary $A$ and $B$ if and only if it is true for $A \oplus B$.  Using the structure theorem for finitely generated abelian groups, we have reduced to the cases $A = \Z$ and $A = \Z_{p^l}$ for a prime $p$ and positive $l$.

Using the inclusion property, we can produce the result for $\Z$ by proving it for $\Q \supseteq \Z$.  The authors of \cite{AHS01} have shown that all symmetric $k$-variable $2$-cocycles over $\Q$ of homogenous degree $n$ are of the form $a \cdot \zeta_k^n$ for $a \in \Q$, and $\zeta_k^n$ has a decomposition into indecomposables by \ref{RationalProjection}.  We then can decompose the $\Z_{p^l}$ case inductively; we have demonstrated a classification for $l = 1$ above, and so we assume that we have accomplished the classification up to some $l \ge 1$.  An $f$ with coefficients in $\Z_{p^{l+1}}$ must be of the form \[f = \sum_i a_i \otimes \beta_i + p^l f'\] for some $f'$, which we can think of as a symmetric $2$-cocycle over $\Z_p$.  We can then again decompose $f'$ into a sum of indecomposables and collect coefficients, giving the desired decomposition of $f$.
\end{proof} 

Since we have now shown that all cocycles take our prescribed form, the only piece of the puzzle left is to actually construct the ring, and we do so in steps.

\begin{thm}\label{RepresentingRingDescription} 
The representing ring for symmetric $k$-variable $2$-cocycles is a tensor of polynomial rings, corrected for torsion.
\end{thm}
\begin{proof}
Fix a homogenous degree $n$ and number of variables $k$.  Then the representing ring for symmetric $2$-cocycles in $k$ variables of this homogenous degree are selected by the coefficients $a_i$ in \ref{ArbitraryRingDecomp}.  If we denote the coefficient of $\zeta_k^n$ as $b_n$ and the coefficient of the polynomial $\beta_i$ in the characteristic $p$ indecomposable basis as $c_{p, i}$, where $i$ ranges over the integers $\{0, \ldots, l_{n, p}\}$, then our representing ring is given by \[L_k^n = \Z[b_n] \otimes \left( \bigotimes_{\substack{\hbox{$p$ prime}, \\ i \in \{1, \ldots, l_{n, p}\}}} \frac{\Z_p[c_{p, i}]}{\langle p c_{p, i} \rangle} \right).\]  Here we drop the zeroth indecomposable basis element because, as noted in \ref{RationalProjection}, $\beta_0 = \zeta_k^n - \sum_{i \ne 0} \beta_i$.

These rings $L^n_k$ can then be tensored together to form $L_k = \bigotimes_n L^n_k$, the representing ring for symmetric $2$-cocycles in $k$ variables.
\end{proof} 

It is worth noting that when $k = 2$ we cover the classic result $L_2 = \Z[b_2, b_3, b_4, \ldots]$, since the innermost tensor product vanishes.

\subsection{For Higher $m$}\label{ForHigherM}

Many of the results in this paper were presented in the context of $\coboundarymap_2$, but in fact yield results for $\coboundarymap_m$ with $m > 2$ as well.

\begin{lem}\label{CocycleGCD} 
If $f$, a symmetric $k$-variable polynomial, is both an $n$-cocycle and an $m$-cocycle, then $f$ is also an $n+m$-cocycle (provided $n + m < k$) and an $|n-m|$-cocycle (provided $n \ne m$).
\end{lem}
\begin{proof}
Assume $n + m < k$, and consider the unreduced sum $\coboundarymap_{n+m} f$.  The first $n+1$ terms of this sum can be reduced to $f(x_0, \ldots, x_{n-1}, x_{n+1}, \ldots, x_k)$ by applying the $n$-cocycle condition.  Then, the remaining $m+2$ summands can be reduced to zero using the $m$-cocycle condition.

Take $m < n$ for simplicity, so $n - m > 0$.  First, we know that $\coboundarymap_n f = 0$, and again we work with the unreduced sum of $\coboundarymap_n f$.  We can replace the last $m+1$ terms of the sum with $f(x_0, \ldots, x_{n-m-1}, x_{n-m+1}, \ldots, x_k)$, and the residual sum forms exactly $\coboundarymap_{n-m} f$.  Transitivity shows that $f \in \ker \coboundarymap_{n-m}$.
\end{proof} 

\begin{cor}\label{CocycleMultiple} 
If $f$ is a symmetric $k$-variable $m$-cocycle, then it is a symmetric $(cm)$-cocycle for all $c$ with $c \cdot m \le k$.
\end{cor}
\begin{proof}
Take $n = m$ and apply the first part of \ref{CocycleGCD} inductively.
\end{proof} 

In particular, \textsection\ref{IntegralProjection} relies on $m = 2$, which allows us to limit the number of parents an image monomial has and make the annihilating set construction (see \ref{AnnihilatorIsAptlyNamed}).  This means that the exhaustivity of our classification here cannot be directly generalized to all $m$ that satisfy $2 \mid m$, though the existence of these gathered cocycles is guaranteed by \ref{CocycleMultiple}.

Using \emph{Mathematica}, we've been able to explore $m > 2$ for relatively small $m$; because $\coboundarymap_m$ only applies to polynomials of dimension at least $m$, the calculations grow unreasonably slow fairly quickly.  Nonetheless, it is our conjecture that the exhaustiveness of gathering is in fact true for all even $m$.  In addition, these gathered cocycles appear to vanish for odd $m$, leaving only the power-of-$p$ symmetrized monomials behind.


\newpage
\appendix

\section{Tables of Modular Additive 2-Cocycles}
\label{TablesOfModulesAdditive2Cocycles}

Here we provide comma delimited lists of modular symmetric 2-cocycles, ordered in rows by degree and in columns by dimension.

\subsection{Characteristic 2}\label{Char2Table}
\begin{equation*}
\begin{array}{r|l|l|l|l|l|}
              & \hbox{dim 2}  &             3    &                4   &                   5   &                      6   \\
\hline
\hbox{deg 2}  & \tau(1, 1)    &             0    &                0   &                   0   &                      0   \\
            3 & \tau(2, 1)    & \tau(1, 1, 1)    &                0   &                   0   &                      0   \\
            4 & \tau(2, 2)    & \tau(2, 1, 1)    & \tau(1, 1, 1, 1)   &                   0   &                      0   \\
            5 & \tau(4, 1)    & \tau(2, 2, 1)    & \tau(2, 1, 1, 1)   & \tau(1, 1, 1, 1, 1)   &                      0   \\
            6 & \tau(4, 2)    & \tau(2, 2, 2),    & \tau(2, 2, 1, 1)   & \tau(2, 1, 1, 1, 1)   & \tau(1, 1, 1, 1, 1, 1)   \\
              &               & \tau(4, 1, 1)    &                    &                       &                          \\
            7 & \tau(6, 1) +  & \tau(4, 2, 1)    & \tau(2, 2, 2, 1) , & \tau(2, 2, 1, 1, 1)   & \tau(2, 1, 1, 1, 1, 1)   \\
              & \tau(5, 2) +  &                  & \tau(4, 1, 1, 1)   &                       &                          \\
              & \tau(4, 3)    &                  &                    &                       &                          \\
            8 & \tau(4, 4)    & \tau(4, 2, 2)    & \tau(2, 2, 2, 2) , & \tau(2, 2, 2, 1, 1) , & \tau(2, 2, 1, 1, 1, 1)   \\
              &               &                  & \tau(4, 2, 1, 1)   & \tau(4, 1, 1, 1, 1)   &                          \\
             9 & \tau(8, 1)    & \tau(4, 4, 1)    & \tau(4, 2, 2, 1)   & \tau(2, 2, 2, 2, 1) , & \tau(2, 2, 2, 1, 1, 1) , \\
               &               &                  &                    & \tau(4, 2, 1, 1, 1)   & \tau(4, 1, 1, 1, 1, 1)   \\
            10 & \tau(8, 2)    & \tau(4, 4, 2) ,  & \tau(4, 2, 2, 2) , & \tau(2, 2, 2, 2, 2) , & \tau(2, 2, 2, 2, 1, 1) , \\
               &               & \tau(8, 1, 1)    & \tau(4, 4, 1, 1)   & \tau(4, 2, 2, 1, 1)   & \tau(4, 2, 1, 1, 1, 1)   \\
            11 & \tau(10, 1) + & \tau(8, 2, 1)    & \tau(4, 4, 2, 1) , & \tau(4, 2, 2, 2, 1) , & \tau(2, 2, 2, 2, 2, 1) , \\
               & \tau(9, 2)  + &                  & \tau(8, 1, 1, 1)   & \tau(4, 4, 1, 1, 1)   & \tau(4, 2, 2, 1, 1, 1)   \\
               & \tau(8, 3)    &                  &                    &                       &                          \\
            12 & \tau(8, 4)    & \tau(4, 4, 4)  , & \tau(4, 4, 2, 2) , & \tau(4, 2, 2, 2, 2) , & \tau(2, 2, 2, 2, 2, 2) , \\
               &               & \tau(8, 2, 2)    & \tau(8, 2, 1, 1)   & \tau(4, 4, 2, 1, 1) , & \tau(4, 2, 2, 2, 1, 1) , \\
               &               &                  &                    & \tau(8, 1, 1, 1, 1)   & \tau(4, 4, 1, 1, 1, 1)   \\
            13 & \tau(12, 1) + & \tau(8, 4, 1)    & \tau(4, 4, 4, 1) , & \tau(4, 4, 2, 2, 1) , & \tau(4, 2, 2, 2, 2, 1) , \\
               & \tau(9, 4)  + &                  & \tau(8, 2, 2, 1)   & \tau(8, 2, 1, 1, 1)   & \tau(4, 4, 2, 1, 1, 1) , \\
               & \tau(8, 5)    &                  &                    &                       & \tau(8, 1, 1, 1, 1, 1)   \\
            14 & \tau(12, 2) + & \tau(8, 4, 2)    & \tau(4, 4, 4, 2) , & \tau(4, 4, 2, 2, 2) , & \tau(4, 4, 2, 2, 1, 1) , \\
               & \tau(10, 4) + &                  & \tau(8, 2, 2, 2) , & \tau(4, 4, 4, 1, 1) , & \tau(8, 2, 1, 1, 1, 1) , \\
               & \tau(8, 6)    &                  & \tau(8, 4, 1, 1)   & \tau(8, 2, 2, 1, 1)   & \tau(4, 2, 2, 2, 2, 2)   \\
            15 & \tau(14, 1) + & \tau(12, 2, 1) + & \tau(8, 4, 2, 1)   & \tau(4, 4, 4, 2, 1) , & \tau(4, 4, 2, 2, 2, 1) , \\
               & \tau(13, 2) + & \tau(10, 4, 1) + &                    & \tau(8, 2, 2, 2, 1) , & \tau(4, 4, 4, 1, 1, 1) , \\
               & \tau(12, 3) + & \tau(9, 4, 2)  + &                    & \tau(8, 4, 2, 1, 1)   & \tau(8, 2, 2, 1, 1, 1)   \\
               & \tau(11, 4) + & \tau(8, 6, 1)  + &                    &                       &                          \\
               & \tau(10, 5) + & \tau(8, 5, 2)  + &                    &                       &                          \\
               & \tau(9, 6)  + & \tau(8, 4, 3)    &                    &                       &                          \\
               & \tau(8, 7)    &                  &                    &                       &                          \\
            16 & \tau(8, 8)    & \tau(8, 4, 4)    & \tau(4, 4, 4, 4) , & \tau(4, 4, 4, 2, 2) , & \tau(4, 4, 2, 2, 2, 2) , \\
               &               &                  & \tau(8, 4, 2, 2)   & \tau(8, 2, 2, 2, 2) , & \tau(4, 4, 4, 2, 1, 1) , \\
               &               &                  &                    & \tau(8, 4, 2, 1, 1)   & \tau(8, 2, 2, 2, 1, 1) , \\
               &               &                  &                    &                       & \tau(8, 4, 1, 1, 1, 1)   \\
\vdots & \vdots & \vdots & \vdots & \vdots & \vdots \\
\end{array}
\end{equation*}

\subsection{Characteristic 3}\label{Char3Table}
\begin{equation*}
\begin{array}{r|l|l|l|l|l|}

  &     \hbox{dim 2}&       3&         4&          5&            6 \\
\hline
 \hbox{deg 2}&\tau(1,1)  &      0 &        0 &         0 &            0 \\
 3&\tau(2,1)  & \tau(1,1,1) &        0 &         0 &            0 \\
 4&\tau(3,1)  & \tau(2,1,1) & \tau(1,1,1,1) &         0 &            0 \\
 5&\tau(3,2)- & \tau(3,1,1) & \tau(2,1,1,1) & \tau(1,1,1,1,1) &            0 \\
  &\tau(4,1)  &             &               &                 &             \\
 6&\tau(3,3)  & \tau(3,2,1)-& \tau(3,1,1,1) & \tau(2,1,1,1,1) & \tau(1,1,1,1,1,1) \\
  &           & \tau(4,1,1) &          &           &             \\
 7&\tau(4,3)- & \tau(3,3,1) & \tau(3,2,1,1)-& \tau(3,1,1,1,1) & \tau(2,1,1,1,1,1) \\
  &\tau(6,1)  &             & \tau(4,1,1,1) &           &             \\
 8&\tau(6,2)+ & \tau(6,1,1)-& \tau(3,3,1,1) & \tau(3,2,1,1,1)-& \tau(3,1,1,1,1,1) \\
  &\tau(4,4)- & \tau(4,3,1)+&               & \tau(4,1,1,1,1) &             \\
  &\tau(7,1)- & \tau(3,3,2) &          &           &             \\
  &\tau(5,3)  &        &          &           &             \\
 9&\tau(6,3)  & \tau(3,3,3) & \tau(6,1,1,1)-& \tau(3,3,1,1,1) & \tau(3,2,1,1,1,1)-\\
  &           &             & \tau(4,3,1,1)+&                 & \tau(4,1,1,1,1,1) \\
  &           &             & \tau(3,3,2,1) &           &             \\
10&\tau(9,1)  & \tau(4,3,3)-& \tau(3,3,3,1) & \tau(6,1,1,1,1)-& \tau(3,3,1,1,1,1) \\
  &           & \tau(6,3,1) &               & \tau(4,3,1,1,1)+&             \\
  &           &             &               & \tau(3,3,2,1,1) &             \\
11&\tau(9,2)- & \tau(9,1,1) & \tau(6,3,1,1)-& \tau(3,3,3,1,1) & \tau(6,1,1,1,1,1)-\\
  &\tau(10,1) &             & \tau(4,3,3,1)+&                 & \tau(4,3,1,1,1,1)+\\
  &           &             & \tau(3,3,3,2) &                 & \tau(3,3,2,1,1,1) \\
12&\tau(9,3)  & \tau(6,3,3),& \tau(3,3,3,3),& \tau(6,3,1,1,1)-& \tau(3,3,3,1,1,1) \\
  &           & \tau(9,2,1)-& \tau(9,1,1,1) & \tau(4,3,3,1,1)+&             \\
  &           &\tau(10,1,1) &               & \tau(3,3,3,2,1) &             \\
13&\tau(12,1)+& \tau(9,3,1) & \tau(4,3,3,3)-& \tau(3,3,3,3,1),& \tau(6,3,1,1,1,1)-\\
  &\tau(10,3)+&             & \tau(6,3,3,1),& \tau(9,1,1,1,1) & \tau(4,3,3,1,1,1)+\\
  &\tau(9,4)  &             & \tau(9,2,1,1)-&                 & \tau(3,3,3,2,1,1) \\
  &           &             &\tau(10,1,1,1) &                 &             \\
14&\tau(12,2)-& \tau(9,3,2)-& \tau(9,3,1,1) & \tau(6,3,3,1,1)-& \tau(3,3,3,3,1,1),\\
  &\tau(13,1)+&\tau(12,1,1)-&               & \tau(4,3,3,3,1)+& \tau(9,1,1,1,1,1) \\
  &\tau(11,3)-&\tau(10,3,1)-&               & \tau(3,3,3,3,2),&             \\
  &\tau(10,4)+& \tau(9,4,1) &               & \tau(9,2,1,1,1)-&             \\
  &\tau(9,5)  &             &               &\tau(10,1,1,1,1) &             \\
15&\tau(9,6)- & \tau(9,3,3) & \tau(6,3,3,3),& \tau(3,3,3,3,3),& \tau(9,2,1,1,1,1)-\\
  &\tau(12,3) &             & \tau(9,3,2,1)-& \tau(9,3,1,1,1) &\tau(10,1,1,1,1,1),\\
  &           &             &\tau(12,1,1,1)-&                 & \tau(6,3,3,1,1,1)-\\
  &           &             &\tau(10,3,1,1)-&                 & \tau(4,3,3,3,1,1)+\\
  &           &             &\tau(9,4,1,1)  &                 & \tau(3,3,3,3,2,1) \\
16&\tau(15,1)-&\tau(12,3,1)-& \tau(9,3,3,1) & \tau(4,3,3,3,3)-& \tau(3,3,3,3,3,1),\\
  &\tau(13,3)-&\tau(9,6,1)+ &               & \tau(6,3,3,3,1),& \tau(9,3,1,1,1,1) \\
  &\tau(12,4)+&\tau(10,3,3)+&               & \tau(9,3,2,1,1)-&             \\
  &\tau(10,6)+& \tau(9,4,3) &               &\tau(10,3,1,1,1)-&             \\
  &\tau(9,7)  &             &               & \tau(9,4,1,1,1)-&             \\
  &           &             &               &\tau(12,1,1,1,1) &             \\
\vdots & \vdots & \vdots & \vdots & \vdots & \vdots \\
\end{array}
\end{equation*}

\subsection{Characteristic 5}\label{Char5Table}
\begin{equation*}
\begin{array}{r|l|l|l|l|l|}
             & \hbox{dim 2}          & 3                        &                        4    &                           5   &                              6   \\
\hline
\hbox{deg 2} & 1 \cdot \tau(1, 1)    &                     0    &                        0    &                           0   &                              0   \\
           3 & 1 \cdot \tau(2, 1)    & 1 \cdot \tau(1, 1, 1)    &                        0    &                           0   &                              0   \\
           4 & 1 \cdot \tau(2, 2)  + & 1 \cdot \tau(2, 1, 1)    & 1 \cdot \tau(1, 1, 1, 1)    &                           0   &                              0   \\
             & 4 \cdot \tau(3, 1)    &                          &                             &                               &                                  \\
           5 & 1 \cdot \tau(3, 2)  + & 1 \cdot \tau(2, 2, 1)  + & 1 \cdot \tau(2, 1, 1, 1)    & 1 \cdot \tau(1, 1, 1, 1, 1)   &                              0   \\
             & 3 \cdot \tau(4, 1)    & 4 \cdot \tau(3, 1, 1)    &                             &                               &                                  \\
           6 & 1 \cdot \tau(5, 1)    & 1 \cdot \tau(2, 2, 2)  + & 1 \cdot \tau(2, 2, 1, 1)  + & 1 \cdot \tau(2, 1, 1, 1, 1)   & 1 \cdot \tau(1, 1, 1, 1, 1, 1)   \\
             &                       & 2 \cdot \tau(4, 1, 1)  + & 4 \cdot \tau(3, 1, 1, 1)    &                               &                                  \\
             &                       & 4 \cdot \tau(3, 2, 1)    &                             &                               &                                  \\
           7 & 1 \cdot \tau(5, 2)  + & 1 \cdot \tau(5, 1, 1)    & 1 \cdot \tau(2, 2, 2, 1)  + & 1 \cdot \tau(2, 2, 1, 1, 1) + & 1 \cdot \tau(2, 1, 1, 1, 1, 1)   \\
             & 2 \cdot \tau(6, 1)    &                          & 2 \cdot \tau(4, 1, 1, 1)  + & 4 \cdot \tau(3, 1, 1, 1, 1)   &                                  \\
             &                       &                          & 4 \cdot \tau(3, 2, 1, 1)    &                               &                                  \\
           8 & 1 \cdot \tau(5, 3)  + & 1 \cdot \tau(5, 2, 1)  + & 1 \cdot \tau(5, 1, 1, 1)    & 1 \cdot \tau(2, 2, 2, 1, 1) + & 1 \cdot \tau(2, 2, 1, 1, 1, 1) + \\
             & 3 \cdot \tau(6, 2)  + & 2 \cdot \tau(6, 1, 1)    &                             & 2 \cdot \tau(4, 1, 1, 1, 1) + & 4 \cdot \tau(3, 1, 1, 1, 1, 1)   \\
             & 3 \cdot \tau(7, 1)    &                          &                             & 4 \cdot \tau(3, 2, 1, 1, 1)   &                                  \\
           9 & 1 \cdot \tau(5, 4)  + & 1 \cdot \tau(5, 2, 2)  + & 1 \cdot \tau(5, 2, 1, 1)  + & 1 \cdot \tau(5, 1, 1, 1, 1)   & 1 \cdot \tau(2, 2, 2, 1, 1, 1) + \\
             & 1 \cdot \tau(7, 2)  + & 2 \cdot \tau(6, 2, 1)  + & 2 \cdot \tau(6, 1, 1, 1)    &                               & 2 \cdot \tau(4, 1, 1, 1, 1, 1) + \\
             & 4 \cdot \tau(6, 3)  + & 2 \cdot \tau(7, 1, 1)  + &                             &                               & 4 \cdot \tau(3, 2, 1, 1, 1, 1)   \\
             & 4 \cdot \tau(8, 1)    & 4 \cdot \tau(5, 3, 1)    &                             &                               &                                  \\
          10 & 1 \cdot \tau(5, 5)    & 1 \cdot \tau(5, 3, 2)  + & 1 \cdot \tau(5, 2, 2, 1)  + & 1 \cdot \tau(5, 2, 1, 1, 1) + & 1 \cdot \tau(5, 1, 1, 1, 1, 1)   \\
             &                       & 2 \cdot \tau(6, 3, 1)  + & 2 \cdot \tau(6, 2, 1, 1)  + & 2 \cdot \tau(6, 1, 1, 1, 1)   &                                  \\
             &                       & 2 \cdot \tau(8, 1, 1)  + & 2 \cdot \tau(7, 1, 1, 1)  + &                               &                                  \\
             &                       & 3 \cdot \tau(5, 4, 1)  + & 4 \cdot \tau(5, 3, 1, 1)    &                               &                                  \\
             &                       & 3 \cdot \tau(6, 2, 2)  + &                             &                               &                                  \\
             &                       & 3 \cdot \tau(7, 2, 1)    &                             &                               &                                  \\
          11 & 1 \cdot \tau(6, 5)  + & 1 \cdot \tau(5, 5, 1)    & 1 \cdot \tau(5, 2, 2, 2)  + & 1 \cdot \tau(5, 2, 2, 1, 1) + & 1 \cdot \tau(5, 2, 1, 1, 1, 1) + \\
             & 3 \cdot \tau(10, 1)   &                          & 2 \cdot \tau(5, 4, 1, 1)  + & 2 \cdot \tau(6, 2, 1, 1, 1) + & 2 \cdot \tau(6, 1, 1, 1, 1, 1)   \\
             &                       &                          & 2 \cdot \tau(6, 2, 2, 1)  + & 2 \cdot \tau(7, 1, 1, 1, 1) + &                                  \\
             &                       &                          & 2 \cdot \tau(7, 2, 1, 1)  + & 4 \cdot \tau(5, 3, 1, 1, 1)   &                                  \\
             &                       &                          & 3 \cdot \tau(6, 3, 1, 1)  + &                               &                                  \\
             &                       &                          & 3 \cdot \tau(8, 1, 1, 1)  + &                               &                                  \\
             &                       &                          & 4 \cdot \tau(5, 3, 2, 1)    &                               &                                  \\
          12 & 1 \cdot \tau(6, 6)  + & 1 \cdot \tau(5, 5, 2)  + & 1 \cdot \tau(5, 5, 1, 1)    & 1 \cdot \tau(5, 2, 2, 2, 1) + & 1 \cdot \tau(5, 2, 2, 1, 1, 1) + \\
             & 1 \cdot \tau(7, 5)  + & 1 \cdot \tau(10, 1, 1) + &                             & 2 \cdot \tau(5, 4, 1, 1, 1) + & 2 \cdot \tau(6, 2, 1, 1, 1, 1) + \\
             & 3 \cdot \tau(11, 1) + & 2 \cdot \tau(6, 5, 1)    &                             & 2 \cdot \tau(6, 2, 2, 1, 1) + & 2 \cdot \tau(7, 1, 1, 1, 1, 1) + \\
             & 4 \cdot \tau(10, 2)   &                          &                             & 2 \cdot \tau(7, 2, 1, 1, 1) + & 4 \cdot \tau(5, 3, 1, 1, 1, 1)   \\
             &                       &                          &                             & 3 \cdot \tau(6, 3, 1, 1, 1) + &                                  \\
             &                       &                          &                             & 3 \cdot \tau(8, 1, 1, 1, 1) + &                                  \\
             &                       &                          &                             & 4 \cdot \tau(5, 3, 2, 1, 1)   &                                  \\
\vdots & \vdots & \vdots & \vdots & \vdots & \vdots \\
\end{array}
\end{equation*}
\newpage

\section{Notation}
\begin{tabular}{ll}
$\N_0 = \N \cup \{0\}$ & $\{0,1,2,3,...\}$ \\
$\Z_p = \Z/p\Z$ & the field with $p$ elements. \\
$S_k$ & the symmetric group on $k$ symbols. \\
$\lambda\cup\mu$ & the multi-index $(\lambda,$ $\mu)$\\
$\lambda \setminus \mu$ & the unique operation satisfying $\left(\lambda \setminus \mu\right) \cup \mu = \lambda$ up to order\\
${\bf x} = (x_1, \ldots, x_k)$ & a $k$-tuple of variables. \\
$A[ {\bf x} ]$ & the ring of polynomials in the $x_i \in {\bf x}$ over $A$. \\
$A\llbracket {\bf x} \rrbracket$ & the ring of power-series in the $x_i \in {\bf x}$ over $A$. \\
$\lambda ! = \prod_i \lambda_i !$ & the partition factorial \\
${n \choose \lambda} = \frac{n!}{\lambda !}$ & the multinomial coefficient of $\lambda$. \\
$\lambda \vdash n$ & $\lambda$ is a partition of $n$. \\
$\zeta_k^n$ & the unique rational $2$-cocycle of degree $n$ in $k$ variables. \\
$\tau \lambda, \tau(\lambda)$ & the monomial symmetric function associated to $\lambda$. \\
$\coboundarymap_m$ & the (additive) $m$-coboundary map. \\
\end{tabular}





\begin{thebibliography}{NameYY}
\bibitem[AHS01]{AHS01}
Ando, M. et al., {\it Elliptic spectra, the Witten genus and the theorem of the cube}.  Inventiones Mathematicae, 2001.
\bibitem[AS01]{AS01}
Ando, M., Strickland, N., {\it Weil pairings and Morava K-theory}.  Topology, 2001.
\bibitem[Bre83]{Bre83}
Breen, L., {\it Fonctions th\^eta et th\'eor\`eme du cube}.  Springer Lectures Notes in Mathematics v. 980, 1983.
\bibitem[Kum852]{Kum852}
Kummer, E. E.  {\it \"Uber die Erg\"anzungss\"atze zu den allgemeinen Reciprocit\"atsgesetzen}. J. Reine Angew. Math., 1852.
\bibitem[Laz55]{Laz55}
Lazard, M. {\it Sur les groupes de Lie formels \`a un param\`etre}.  Bull. Soc. Math., 1955.
\bibitem[Tate57]{Tate57}
Tate, J. {\it Homology of Noetherian Rings and Local Rings}.  Ill. J. of Math., 1957.
\end{thebibliography}
\end{document}